\documentclass[a4paper,10pt]{amsart}%

\usepackage{amsfonts}
\usepackage{amssymb}
\usepackage[utf8]{inputenc}
\usepackage{amsmath}
\usepackage{xcolor}
\usepackage{graphicx}%
\providecommand{\U}[1]{\protect\rule{.1in}{.1in}}
\usepackage[pagebackref=true, bookmarksopen=true, colorlinks=true, linkcolor=red, citecolor=blue]{hyperref}
\numberwithin{equation}{section}

\newtheorem{thm}{Theorem}[section]

\newtheorem{lemma}[thm]{Lemma}

\newtheorem{prop}[thm]{Proposition}

\newtheorem{remark}[thm]{Remark}
\newtheorem{remarks}[thm]{Remarks}
\newtheorem{definition}{Definition}[section]

\newcommand{\R}{\mathbb{R}}

\newcommand{\N}{\mathbb{N}}

\newcommand{\LL}{\mathcal{L}}

\newcommand{\GG}{\mathcal{G}}

\newcommand{\EE}{\mathcal{E}}
\newcommand{\NN}{\mathcal{N}}

\begin{document}
\title[Asymptotic Behavior of KdV-KdV System]{Asymptotic behavior of Boussinesq system of KdV-KdV type}
\author[Capistrano--Filho]{R. A. Capistrano--Filho}
\address{Departamento de Matemática, Universidade Federal de Pernambuco, 50740-545, Recife-PE, Brazil.}
\email{capistranofilho@dmat.ufpe.br}
\author[Gallego]{F. A. Gallego}
\address{Departamento de Matem\'aticas y Estad\'istica, Universidad Nacional de Colombia - Sede Manizales, Colombia}
\email{fagallegor@unal.edu.co}
\date{}

\begin{abstract}
This work deals with the local rapid exponential stabilization for a Boussinesq system of KdV-KdV type introduced by J. Bona, M. Chen and J.-C. Saut. This is a model for the motion of small amplitude long waves on the surface of an ideal fluid. Here, we will consider the Boussinesq system of KdV-KdV type posed on a finite domain, with homogeneous Dirichlet--Neumann boundary controls acting at the right end point of the interval. Our goal is to build suitable integral transformations to get a feedback control law that leads to the stabilization of the system. More precisely, we will prove that the solution of the closed-loop system decays exponentially to zero in the $L^2(0,L)$--norm and the decay rate can be tuned to be as large as desired if the initial data is small enough.
\end{abstract}

\maketitle

\section{Introduction}
\subsection{Setting of the problem}
The classical Boussinesq systems were first derived by Boussinesq in
\cite{boussinesq1}, to describe the two-way propagation of small amplitude,
long wave length gravity waves on the surface of water in a canal. These
systems and their higher-order generalizations also arise when modeling the
propagation of long-crested waves on large lakes or on the ocean and in other
contexts. Recently, in \cite{bona-chen-saut}, the authors derived a four-parameter
family of Boussinesq systems to describe the motion of small amplitude long
waves on the surface of an ideal fluid under the gravity force and in
situations where the motion is sensibly two dimensional. More precisely, they
studied a family of systems of the form%
\begin{equation}
\left\{
\begin{array}
[c]{l}%
\eta_{t}+w_{x}+(   \eta w)  _{x}+aw_{xxx}-b\eta_{xxt}=0\text{,}\\
w_{t}+\eta_{x}+ww_{x}+c\eta_{xxx}-dw_{xxt}=0\text{.}%
\end{array}
\right.  \label{int_29e_crp_1}%
\end{equation}
In (\ref{int_29e_crp_1}), $\eta$ is the elevation from the equilibrium position, and
$w=w_{\theta}$ is the horizontal velocity in the flow at height $\theta h$,
where $h$ is the undisturbed depth of the liquid. The parameters $a$, $b$,
$c$, $d$, that one might choose in a given modeling situation, are required to
fulfill the relations%
\begin{equation}
a+b=\frac{1}{2}\left(   \theta^{2}-\frac{1}{3}\right)  \text{, \ \ \ }%
c+d=\frac{1}{2}(   1-\theta^{2})  \geq0 \text{, \ \ \ }\theta
\in\left[  0,1\right], \label{int_30e_crp_1}%
\end{equation}
where $\theta\in\left[  0,1\right]  $ specifies which horizontal velocity the
variable $w$ represents (cf. \cite{bona-chen-saut}).
Consequently,
\[
a+b+c+d=\frac{1}{3}.
\]
As it has been proved in \cite{bona-chen-saut}, the initial value problem for the linear system associated with
\eqref{int_29e_crp_1} is well-posed on $\mathbb R$ if either $C_1$ or $C_2$ is satisfied, where
\begin{eqnarray*}
(C_1)&& b,d\ge 0,\ a\le 0,\ c\le 0;\\
(C_2)&& b,d\ge 0, \ a=c>0. 
\end{eqnarray*}
When $b=d=0$ and $(C_2)$ is satisfied, then necessarily $a=c=1/6$. Nevertheless, the scaling $x\to x/\sqrt{6}$, $t\to t/\sqrt{6}$ gives an system equivalent to \eqref{int_29e_crp_1} for which $a=c=1$, namely
\begin{equation}
\label{b1}
\begin{cases}
\eta_t + w_x+w_{xxx}+  (\eta w)_x= 0, & \text{in} \,\, (0,L)\times (0,+\infty) ,\\
w_t +\eta_x +\eta_{xxx} +ww_x=0,  & \text{in} \,\, (0,L)\times (0,+\infty), \\
\eta(x,0)= \eta_0(x), \quad w(x,0)=  w_0(x), & \text{in} \,\, (0,L),
\end{cases}
\end{equation}
which is the so-called \textit{Boussinesq system of Korteweg-de Vries---Korteweg-de Vries type}. 

Therefore, the interest of this work is to give a positive answer for the following stabilization problem:
\vglue 0.2 cm
\noindent
\textit{{\bf Problem $\mathcal{A}$:}
Can one find a linear feedback control law $$(f(t),g(t)) = F[(\eta( \cdot,t),w(\cdot,t))], \quad t \in (0,\infty),$$
such that the closed-loop system \eqref{b1} with boundary condition
\begin{equation}\label{b1.1}
\begin{cases}
\eta(0,t)=0,\,\,\eta(L,t)=0,\,\,\eta_{x}(0,t)=f(t),&t \in (0,\infty), \\
w(0,t)=0,\,\,w(L,t)=0,\,\,w_{x}(L,t)=g(t),& t \in (0,\infty)
\end{cases}
\end{equation}
 is exponentially stable?}
\vglue 0.1 cm

As we know, there are some natural methods that may give us a positive answers of the Problem $\mathcal{A}$, e.g., the so-called “Gramian approach” (see \cite{komornik1997,slemrod,urquiza2005} and the reference therein for more details), the Lyapunov function method (see, for instance, \cite{coron2007}), and, finally, the backstepping method, that is now a standard method for finite dimensional control systems (see, e.g., \cite{coron2007,krstic1995,krstic2008,Sontag}). The first adaptations of the backstepping method to control systems modeled by partial differential equations were given in \cite{coron1998} and \cite{liu2000},  by using a Volterra transformation \cite{boskovi2001}.


\subsection{State of art} 
In this paper, we will try to apply the backstepping method (see \cite{krstic2008} for a systematic introduction of this method) to design the feedback control law. This method was successfully applied by Coron \textit{et. al.} in \cite{coron2014,coron2015} to study the rapid stabilization of the Korteweg–de Vries (KdV) and the Kuramoto–Sivashinsky (K--S) equations, respectively.

More precisely, which concerning of KdV equation, in \cite{coron2014}, the authors studied the KdV equation on a bounded domain $(0,L)$
\begin{equation}
\left\{
\begin{array}
[c]{lll}%
u_{t}+uu_x+u_{x}+u_{xxx}=0 &  & \text{in }(0,L)\times (0,+\infty)  \text{,}\\
u(0,t)=u(L,t)=0 &  & \text{on }( 0,+\infty),\\
u_x(L,t)=h(t) & & \text{on } ( 0,+\infty),
\end{array}
\right.  \label{kdv_1}%
\end{equation}
where the function $h(t)=F_{\lambda}(u(t))$ is the feedback law designed to ensure the exponential stability of the system with a decay rate equal to $\lambda$.  This decay rate can be chosen as large as desired, which is called a rapid stabilization result. They consider the following problem:
\vglue 0.2 cm
\noindent
\textit{{\bf Problem $\mathcal{B}$:} Let $\lambda>0$. Does there exist a linear feedback control $F_{\lambda}:L^2(0,L)\to\mathbb{R}$ such that, for some $\delta>0$, every solution $u$ of \eqref{kdv_1} with $h(t) = F_{\lambda}(u(\cdot,t))$ satisfies $$||u(\cdot,t)||_{L^2(0,L)}\leq Ce^{-\lambda t}||u(\cdot,0)||_{L^2(0,L)},$$
for some $C>0$, provided that $||u(\cdot,0)||_{L^2(0,L)}\leq\delta$?}
\vglue 0.1 cm
A positive answer to this question was given in \cite{coron2014} (see also \cite{coron2015} for K--S equation) using a \textit{modified backstepping method} and the following result was obtained.

\vglue 0.2 cm
\noindent
\textit{{\bf Theorem A }(Coron \textit{et al.} \cite{coron2014}) For every $\lambda>0$, there exist a continuous linear feedback control law $F:L^2(0,L)\to\mathbb{R}$, and positive constants  $r>0$ and $C>0$ such that, for every $u^0\in L^2(0,L)$ satisfying $||u^0||_{L^2(0,L)}\leq r$, the solution $v$ of \eqref{kdv_1}, with $h(t):=F(u(\cdot,t))$ satisfying the initial condition $u(\cdot,0)=u^0(\cdot)$, is defined on $[0,+\infty)$ and satisfies $$||u(\cdot,t)||_{L^2(0,L)}\leq Ce^{-\frac{\lambda}{2}t}||u(\cdot,0)||_{L^2(0,L)}, \quad \text{ for every } \quad t\geq0.$$}

The main difficulty to establish Theorem A is the fact that the linear system is known to be non-controllable (and consequently non-stabilizable) if the length of the interval $L$ belongs to a set of critical values $\NN$ (see for instance \cite{rosier}). The authors looked for an integral transform
\begin{equation}\label{vt}
w(x,t)= u(x,t) - \int_0^Lk(x,y)u(y,t)dy
\end{equation}
with $k=k(x,y)$ chosen such that $u=u(x,t)$ is a solution of \eqref{kdv_1} if and only if $w=w(x,t)$ is a solution of 
\begin{equation}\label{estintro}
\begin{cases}
w_{t}+w_x+w_{xxx}+\lambda w=-uu_x-\frac{1}{2}\int_0^L k_y(x,y)u^2(y,t)dy, \\
w(0,t)=w(L,t)=w_x(L,t)=0.
\end{cases}
\end{equation}
Thus, they proved that the system \eqref{estintro} is locally exponentially stable with a decay rate equal to $\lambda$. This result follows for \eqref{kdv_1} if the integral transform \eqref{vt} exists and is invertible and if the kernel function $k=k(x,y)$ satisfies a partial differential equation with a Dirac measure as a source term. Is important to see that this result holds if the length $L$ is not critical.

\vglue 0.3cm

Now, we will come back to the stabilization properties of the system \eqref{b1} on a bounded domain. As far we know, the KdV--KdV system is expected to admit global
solutions on $\mathbb{R}$, and it also possesses good control properties on
the torus \cite{micu2}. However, there are few results concerning to the bounded domains. The unique result in the literature is due to Pazoto and Rosier in \cite{pazoto2008}. They investigated the asymptotic behavior of the solutions of the system \eqref{b1} satisfying the boundary conditions%
\begin{equation}
\left\{
\begin{array}
[c]{lll}%
w(0,t)  =w_{xx}(0,t)  =0\text{ } &  & \text{on $(0,T)$,}\\
w_{x}(0,t)  =\alpha_{0}\eta_{x}(0,t)  &  & \text{on $(0,T)$,}\\
w(L,t)  =\alpha_{2}\eta(L,t)  &  & \text{on $(0,T)$,}\\
w_{x}(L,t)  =-\alpha_{1}\eta_{x}(L,t)  &  & \text{on $(0,T)$,}\\
w_{xx}(L,t)  =-\alpha_{2}\eta_{xx}(L,t)  &  &\text{on $(0,T)$}%
\end{array}
\right.  \label{int_32e_crp_1}%
\end{equation}
and initial conditions%
\begin{equation}%
\begin{array}
[c]{lll}%
\eta(x,0) =\eta_{0}(x)  \text{, \ }w(x,0)  =w_{0}(   x)  &  & \text{on }(   0,L)
\text{.}%
\end{array}
\label{int_33e_crp_1}%
\end{equation}
In (\ref{int_32e_crp_1}), $\alpha_{0}$, $\alpha_{1}$ and $\alpha_{2}$ denote some
nonnegative real constants.

Under the above boundary conditions, they observed that the derivative of the energy associated to the system (\ref{b1}), with boundary
conditions (\ref{int_32e_crp_1})-(\ref{int_33e_crp_1}), satisfies%
\[
\frac{dE}{dt}    =-\alpha_{2}\left\vert \eta(   L,t)  \right\vert
^{2}-\alpha_{1}\left\vert \eta_{x}(   L,t)  \right\vert ^{2}%
-\alpha_{0}\left\vert \eta_{x}(   0,t)  \right\vert ^{2}
  -\frac{1}{3}w^{3}(   L,t)  -\int_{0}^{L}(   \eta w)
_{x}\eta dx,
\]
where%
\[
E(   t)  =\frac{1}{2}\int_{0}^{L}(   \eta^{2}+w^{2})
dx\text{.}%
\]

This indicates that the boundary conditions play the role of a feedback
damping mechanism, at least for the linearized system. Therefore, the
following questions arise:

\vglue 0.2 cm
\noindent
\textit{{\bf Problem $\mathcal{C}$:} Does $E(t)\rightarrow0$, as $t\rightarrow+\infty?$  If it is the case, can we give the decay rate?}
\vglue 0.1 cm

The problem might be easy to solve when the underlying model has a intrinsic
dissipative nature. Moreover, in the context of coupled systems, in order to
achieve the desired decay property, the damping mechanism has to be designed
in an appropriate way in order to capture all the components of the system.
The main result of \cite{pazoto2008} provides a positive answer to those questions.

\vglue 0.2 cm
\noindent
\textit{{\bf Theorem B }(Pazoto \textit{et al.} \cite{pazoto2008}) 
Assume that $\alpha_{0}\geq0$, $\alpha_{1}>0$ and
$\alpha_{2}=1$. Then there exist some numbers $\rho>0$, $C>0$ and $\mu>0$ such
that for any $(   \eta_{0},w_{0})  \in(   L^{2}(   I)
)  ^{2}$ with%
\[
\left\Vert (   \eta_{0},w_{0})  \right\Vert _{(   L^{2}( 
I)  )  ^{2}}\leq\rho\text{,}%
\]
the system (\ref{b1})-(\ref{int_33e_crp_1}) admits a unique solution%
\[
(   \eta,w)  \in C(   \mathbb{R}^{+};(   L^{2}( 
I)  )  ^{2})  \cap C(   \mathbb{R}^{+\ast};( 
H^{1}(   I)  )  ^{2})  \cap L^{2}( 
0,1  ;(   H^{1}(   I)  )  ^{2})  \text{,}%
\]
which fulfills%
\[
\left\Vert (   \eta,w)  (   t)  \right\Vert _{( 
L^{2}(   I)  )  ^{2}}\leq Ce^{-\mu t}\left\Vert ( 
\eta_{0},w_{0})  \right\Vert _{(   L^{2}(   I)  )
^{2}}\text{, }\forall t\geq0\text{,}%
\]%
\[
\left\Vert (   \eta,w)  (   t)  \right\Vert _{( 
H^{1}(   I)  )  ^{2}}\leq C\frac{e^{-\mu t}}{\sqrt{t}%
}\left\Vert (   \eta_{0},w_{0})  \right\Vert _{(   L^{2}( 
I)  )  ^{2}}\text{, }\forall t>0\text{.}%
\]}

In our case, we propose to use the ideas contained in \cite{coron2014} to obtain a positive answer for the Problem $\mathcal{A}$. However, first, we need to know a answer for the the following exact controllability problem.

\vglue 0.2 cm
\noindent
\textit{{\bf Problem $\mathcal{D}$:} Given $T>0$ and $(\eta_0,w_0)$, $(\eta_T,w_T)$ in certain space, can one find appropriate $f(t)$ and $g(t)$, in a certain space, such that the corresponding solution $(\eta,w)$ of the linearized system
\begin{equation}
\left\{%
\begin{array}
[c]{lll}%
\eta_{t}+w_{x}+ w_{xxx}=0 &  & \text{in }( 
0,L)  \times(   0,T)  \text{,}\\
w_{t}+\eta_{x} + \eta_{xxx}=0 &  & \text{in }(   0,L)
\times(   0,T)  \text{,}%
\end{array}
\right.  \label{int_34e}%
\end{equation}
with the boundary conditions
\begin{equation}
\left\{%
\begin{array}
[c]{lll}%
\eta(0,t) =\eta(L,t)  =0\text{, } \eta_{x}(0,t)  =f(t)  &  & \text{on }(   0,T)
\text{,}\\
w(0,t)  =w(L,t)
=0 \text{, }w_{x}( 
L,t)  =g(   t)  &  & \text{on }(   0,T)
\end{array}
\right.  \label{int_35e}%
\end{equation}
and the initial conditions
\begin{equation}%
\begin{array}
[c]{lll}%
\eta( x,0)  =\eta_{0}(   x)  \text{, \ }w( 
x,0)  =w_{0}(   x)  &  & \text{in }(   0,L)
\text{.}%
\end{array}
\label{int_36e}
\end{equation}
satisfies $\eta(\cdot,  T)  =\eta_{T}$ and $w(\cdot,T)
=w_{T}$?}
\vglue 0.1 cm

More recently, in \cite{capistrano2016} (see also \cite{capis_thesis}), the exact boundary controllability of the linear system Boussinesq of KdV--KdV type was studied. It was discovered that whether the associated linear system is exactly controllable or not depends on the length of the spatial domain. In the context of equations that possess a hyperbolic structure, recent results deals with systems of two wave-type equations, only one of them being directly damped. More precisely, the following result was obtained for the system \eqref{int_34e}-\eqref{int_36e}.


\vglue 0.2 cm
\noindent
\textit{{\bf Theorem C }(Capistrano--Filho \textit{et al.} \cite{capistrano2016}) Let 
\begin{equation}
\mathcal{N}:=\left\{  \frac{2\pi}{\sqrt{3}}\sqrt{k^{2}+kl+l^{2}}%
\,:k,\,l\,\in\mathbb{N}^{\ast}\right\}. \label{criticalrapid}%
\end{equation}
For any $T>0$, $L\in(   0,+\infty)  \backslash\mathcal{N}$,
$(   \eta_{0},w_{0}) \in  [H^{-1}(0,L)]^{2}$ and $(\eta_{T},w_{T})  \in  [H^{-1}(0,L)]^{2}$, 
 there exist controls $(f(t),g(t))\in [L^{2}(0,T)]^2$ such that the solution $(   \eta,w)  \in C^0([0,T],[H^{-1}(0,L)]^2)$ of \eqref{int_34e}-\eqref{int_36e}, 
satisfies $\eta(\cdot,T)  =\eta_{T}$ and $w(\cdot,T)=w_{T}$.}

\vglue 0.1 cm
As in the case of the KdV equation \cite[Lemma 3.5]{rosier} when $L\in\mathcal{N}$, the linear system \eqref{int_34e}-\eqref{int_36e} is not controllable\footnote{There exists a finite-dimensional subspace of $L^2(0,L)$, denoted by $\mathcal{M}=\mathcal{M}(L)$, which is unreachable from $0$ for the linear system.}. To prove Theorem C, the authors used the classical duality approach based upon the Hilbert Uniqueness Method (H.U.M.)  due to J.-L. Lions \cite{lions}, which reduces the exact controllability of the system to some observability inequality to be proved for the adjoint system. Then, to establish the required observability inequality, was used the compactness-uniqueness argument due to J.-L. Lions \cite{lions1} and some multipliers, which reduces the analysis to study a spectral problem. The spectral problem is finally solved by using a method introduced in \cite{rosier}, based  on the use of complex analysis, namely, the Paley-Wiener theorem.


\subsection{Main result and notations}
With all of these results in hands, we are now in position to present our main result. In this paper, considerations about the local rapid stabilization of the following system
\begin{equation}
\label{n1'''}
\begin{cases}
\eta_t + w_x+w_{xxx}+(\eta w)_x= 0, & \text{in} \,\, (0,L)\times (0,+\infty),\\
w_t +\eta_x +\eta_{xxx} +ww_x=0,  & \text{in} \,\, (0,L)\times (0,+\infty), \\
\eta(x,0)= \eta_0(x), \quad w(x,0)=  w_0(x), & \text{in} \,\, (0,L),
\end{cases}
\end{equation}
with the boundary conditions
\begin{equation}\label{n1.1''}
\begin{cases}
\eta(0,t)=0,\,\,\eta(L,t)=0,\,\,\eta_{x}(0,t)=f(t)& \text{on} \,\, (0,T),\\
w(0,t)=0,\,\,w(L,t)=0,\,\, w_{x}(L,t)=g(t)& \text{on} \,\, (0,T)
\end{cases}
\end{equation}
are given. However, before to present our main result, we remark that some restriction on the length $L$, of the domain are needed. In this paper, unless otherwise specified, we always keep the assumption that $L\in(   0,+\infty)  \backslash\mathcal{N}$, where $\mathcal{N}$ is defined by \eqref{criticalrapid}. 

Let us define the following Hilbert space $$X_s := [H^s(0,L)]^2, \quad \text{for}\quad  s\in \R$$ and 
\begin{align*}
&\overline{X}_0 := X_0, \\
&\overline{X}_3 := \left\lbrace (\eta, w) \in [H^3(0,L)\cap H^1_0(0,L)]^2: \eta_x(0)=w_x(L)=0\right\rbrace, \\
&\overline{X}_{3\theta} := [\overline{X}_0,\overline{X}_3]_{[\theta]}, \quad \text{for $0<\theta <1$,}
\end{align*}
where $[X_0,D(A)]_{[\theta]}$ denote the Banach space obtained by the complex interpolation method (see e.g. \cite{bergh}). It is easily seen that
\begin{align*}
\overline{X}_1 &:= H_0^1(0,L)\times H^1_0(0,L), \\
\overline{X}_2 &:= \left\lbrace (\eta,w)\in [H ^2(0,L)\cap H^1_0(0,L) ]^2: \eta_x(0)=w_x(L)=0\right\rbrace.
\end{align*}
In addition, the space $X_{-s}=(X_s)'$ is defined as the dual of $X_s$ with respect to the pivot space. We also introduce the following integrals transforms $K$ and $S$ defined in $L^2(0,L)$ given by
\begin{equation}\label{deftransf'}
(Kv)(x):=\int_0^L k(x,y)v(y)dy \quad \text{and} \quad (Sv)(x):=\int_0^L s(x,y)v(y)dy, \quad \text{for all $v \in L^2(0,L)$,}
\end{equation}
where $(k,s)$ is the solution of stationary problem
\begin{equation}\label{esta1'}
\begin{cases}
k_{yyy}+k_y+k_{xxx}+k_x+\lambda s=0, & \text{in $(0,L)\times (0,L)$}, \\
s_{yyy}+s_y+s_{xxx}+s_x+\lambda k=\lambda \delta(x-y), & \text{in $(0,L)\times (0,L)$},
\end{cases}
\end{equation}
with boundary condition 
\begin{equation}\label{esta1.1'}
\left\lbrace\begin{array}{l l }
k(x,0)=k(x,L)=k(0,y)=k(L,y)=k_y(x,0)=k_y(x,L)=0, &   \text{on $(0,L)$}, \\
s(x,0)=s(x,L)=s(0,y)=s(L,y)=s_y(x,0)=s_y(x,L)=0, &   \text{on $(0,L)$}, 
\end{array}\right.
\end{equation}
where $\lambda \in \R\setminus \{0\}$ and $\delta(x-y)$ denotes the Dirac measure on the diagonal of the square $[0,L]\times [0,L]$. The definition of a solution to \eqref{esta1'}-\eqref{esta1.1'} is given in Section \ref{sec21}. With this system in hands, we are able to prove the following assertions:
\begin{itemize}
\item[a.] System \eqref{esta1'}-\eqref{esta1.1'} has a unique solution;
\item[b.] The operators $(I- (K+S))^{-1}$ and $(I- (K-S))^{-1}$ exist and it is a continuous operator in $L^2(0,L)$, that is, $(I- (K+S))^{-1}$ and $(I- (K-S))^{-1}$ belongs to $\mathcal{L}(L^2(0,L);\mathbb{R})$;
\item[c.] Assuming that $L\in(   0,+\infty)  \backslash\mathcal{N}$, if we define the feedback law $F(\cdot)=(F_1(\cdot),F_2(\cdot))$ by $$f(t):=F_1(\eta(\cdot, t),w(\cdot, t)) =\int_0^L \left[ k_x(0,y)\eta(y,t)+s_x(0,y)w(y,t)\right]dy \text{ on } (0,T)$$ and  $$g(t):=F_2(\eta(\cdot,t),w(\cdot,t)) = \int_0^L \left[ k_x(L,y)w(y,t)+s_x(L,y)\eta(y,t)\right]dy  \text{ on } (0,T)$$
then, for the solution $(\eta,w)\in C(\R^+;\overline{X}_2)$ of \eqref{n1'''}-\eqref{n1.1''}, one has $$||(\eta(t),w(t))||_{X_0}\leq C\left(||(I-K)\eta(t)-Sw(t)||_{L^2(0,L)}+||(I-K)w(t)-S\eta(t)||_{L^2(0,L)}\right),$$
for some $C:=C(K,S)>0$ depending on the operators $K$ and $S$.
\end{itemize}
\smallskip

Thus, in order to prove the local exponential stability of \eqref{n1'''}-\eqref{n1.1''} and since the \textit{critical set} of the Boussinesq system of KdV--KdV type is defined by \eqref{criticalrapid}, the above statements are the key to prove the main result of this paper, which can be stated as follows.
\begin{thm}\label{main_int}
 Let $T>0$ and $L\in(   0,+\infty)  \backslash\mathcal{N}$.  For every $\lambda>0$, there exist a continuous linear feedback control law $$F:=(F_1,F_2):L^2(0,L)\times L^2(0,L)\to\mathbb{R}\times\mathbb{R},$$ and positive constants $\rho\in(0,+\infty)$ and $C:=C(K,S)>0$, depending on the operators $K$ and $S$, defined by \eqref{deftransf'}, such that, for every $(\eta_0,w_0)\in \overline{X}_2$ with $$\|(\eta_0,w_0)\|_{\overline{X}_2}<\rho,$$ the solution $(\eta,w)$ of \eqref{n1'''}-\eqref{n1.1''} belongs to space $C([0,T]; \overline{X}_2)$ and satisfies
\begin{equation*}
\|(\eta(t),w(t))\|_{X_0} \leq C e^{-\frac{\lambda}{2} t} \|(\eta_0,w_0)\|_{X_0}, \quad  0\leq  t \leq T. 
\end{equation*}
\end{thm}

\smallskip

\begin{remarks}
The following remarks are now in order.
\smallskip

\begin{itemize}
\item[(i.)] In general to propose some stabilization result using standard methods is necessary that the derivative of the energy be negative, see for instance \cite{pazoto2008}. However, using this approach, namely, \textit{backstepping method}, the mechanism of damping does not give us any signal of the energy even for the linear system.
\smallskip

\item[(ii.)] \textit{Backstepping method} is interesting due of the fact that we can deal directly with nonlinear problem instead of first to prove a result for the linear problem and after, using the fixed point argument, to extend for the nonlinear one.
\smallskip

\item[(iii.)] In this work, we can not able to deal with the problem with only one control acting on the boundary conditions, that is, $f(t)=0$ or $g(t)=0$ in \eqref{n1.1''}.  Its looks a interesting problem and we detail the difficulties at the end of this work (see Section \ref{FC}).
\end{itemize}
\end{remarks}

\smallskip
Thus, the plan of the present paper is as follows.

\smallskip
-- Section \ref{lc} is devoted to study the well-posedness of the system \eqref{b1}. 

\smallskip
-- In the Section \ref{st_p}, we will prove the existence and uniqueness of solution of the stationary problem \eqref{esta1'}-\eqref{esta1.1'}. 

\smallskip
-- The proof of our main result, Theorem \ref{main_int}, is made in the Section \ref{rs_p}. 

\smallskip
-- Section \ref{FC} is devoted to some remarks and related problems.


\section{Well-posedness}\label{lc}
In this section, we explain what we mean by a solution of \eqref{b1} and we prove that the non-homogeneous linear system is well-posed.
\subsection{Well-posedness:  Linear homogeneous system} 
The result below is related with the existence of solutions of the following linear system
\begin{equation}\label{homo1}
\begin{cases}
\eta_t + w_x+w_{xxx}= 0, & \text{in} \,\, (0,L)\times (0,T),\\
w_t +\eta_x +\eta_{xxx} =0,  & \text{in} \,\, (0,L)\times (0,T), \\
\eta(0,t)=\eta(L,t)=\eta_{x}(0,t)=0,&\text{on} \,\, (0,T), \\
w(0,t)=w(L,t)=w_{x}(L,t)=0, &\text{on}dg \,\, (0,T), \\
\eta(x,0)= \eta_0(x), \quad w(x,0)=  w_0(x), & \text{on} \,\, (0,L), \\
\end{cases}
\end{equation}
and can be found in \cite[Proposition 2.1]{capistrano2016}, thus we will omit the proof.
\begin{thm}\label{teo1rapid}
Let $(\eta_0,w_0) \in X_0$.  Let $A(\varphi,\psi)=(-\psi'-\psi''',-\varphi'-\varphi''')$ with domain $D(A)=\overline{X}_3$.
Then, there exists a unique mild solution of \eqref{homo1} such that $$ (\eta,w)=S(\cdot)(\eta_0,w_0) \in C(\R^+; X_0),$$ where $S(\cdot)$ is a group of isometries in $X_0$ generated by operator $A$. Moreover, if $(\eta_0,w_0) \in D(A)$, then \eqref{homo1} has a unique (classical) solution $(\eta,w)$ belongs to $C(\R^+ ;D(A))\cap C^1(\R^+;X_0)$.
\end{thm}

Using Theorem \ref{teo1rapid} combined with some interpolation argument between $\overline{X}_0$ and $\overline{X}_3$, we infer 
for any $s \in [0, 3]$, there exists a constant $C_s > 0$ such that for any $(\eta_0, w_0) \in \overline{X}_s$, the solution $(\eta, w)$ of \eqref{homo1} satisfies $(\eta,w) \in C(\R;\overline{X}_s)$ and
\begin{equation}\label{new1}
\|(\eta(t),w(t))\|_{\overline{X}_s} \leq  C_s\|(\eta_0, w_0)\|_{\overline{X}_s}, \quad \forall t \in \R,
\end{equation}
for instance, see \cite{capistrano2016} for more details. 


\subsection{Well-posedness: Nonlinear system}  We are now in position of to prove the well-posedness for the nonlinear system \eqref{n1'''}-\eqref{n1.1''}, where $f(t):=F_1(\eta(t),w(t))$ and $g(t):= F_2(\eta(t),w(t))$, with $F_i$ belongs to $\LL(X_0;\mathbb{R})$, $i=1,2$. The following theorem can be proved.
\begin{thm}\label{nonlinearexitence1}
Let $F_i:X_0\longrightarrow \R$ be a continuous linear map for $i=1,2$ and $T>0$. Then, there exists $\rho>0$ such that,  for every $(\eta_0,w_0) \in \overline{X}_2$ satisfying  
\begin{equation*}
\|(\eta_0,w_0)\|_{ \overline{X}_2 }< \rho,
\end{equation*}
 there exists a unique solution $(\eta, w) \in C([0,T]; \overline{X}_2)$ of \eqref{n1'''}-\eqref{n1.1''} with $f(t):=F_1(\eta(t),w(t))$ and $g(t):=F_2(\eta(t),w(t))$. Moreover
\begin{align*}
\|(\eta,w)\|_{C([0,T]; \overline{X}_2)} \leq C \|(\eta_0,w_0)\|_{ \overline{X}_2}
\end{align*}
for some positive constant $C=C(T)$.
\end{thm}

Before to present the proof of the Theorem \ref{nonlinearexitence1}, it is necessary to establish some definition to recall how the solution of
the problem \eqref{n1'''}-\eqref{n1.1''} can be defined.
\begin{definition}\label{def}
Given $T>0$, $(\eta_0,w_0) \in \overline{X}_2$, $(h_1,h_2) \in L^2(0,T;X_{-1})$ and $(f,g) \in  [L^2(0,T)]^2$. Consider the non-homogeneous system 
\begin{equation}
\label{nhs1}
\begin{cases}
\eta_t + w_x+w_{xxx}= h_1, & \text{in} \,\, (0,L)\times (0,T),\\
w_t +\eta_x +\eta_{xxx} =h_2,  & \text{in} \,\, (0,L)\times (0,T), \\
\eta(0,t)=0,\,\,\eta(L,t)=0,\,\,\eta_{x}(0,t)=f(t),& \text{on} \,\, (0,T), \\
w(0,t)=0,\,\,w(L,t)=0,\,\,w_{x}(L,t)=g(t), & \text{on} \,\, (0,T), \\
\eta(x,0)= \eta_0(x), \quad w(x,0)=  w_0(x), & \text{on} \,\, (0,L).
\end{cases}
\end{equation}
A solution of the problem \eqref{nhs1} is a function $(\eta,w)$ in $C([0,T],\overline{X}_2)$ such that, for any $\tau \in [0,T]$ and $(\varphi_\tau,\psi_\tau) \in \overline{X}_2$, the following identity holds
\begin{multline}\label{transp1}
\left((\eta(\tau),w(\tau)),(\varphi_\tau,\psi_\tau)\right)_{\overline{X}_2} = \left((\eta_0,w_0),(\varphi(0),\psi(0))\right)_{\overline{X}_2}  -\int_0^\tau f(t)\psi_x(0,t)dt \\
+\int_0^\tau g(t)\varphi_x(L,t)dt+ \int_0^\tau \left\langle \varphi(t), h_1(t)\right\rangle_{H^1_0\times H^{-1}} dt+\int_0^\tau \left\langle \psi (t),h_2(t)\right\rangle_{H^{1}_0\times H^{-1}}  dt,
\end{multline}
where $\left(\cdot,\cdot\right)_{\overline{X}_2}$ is the inner product of $\overline{X}_2$, $\left\langle\cdot,\cdot\right\rangle$ is the duality of two spaces and $(\varphi,\psi)$ is the solution of 
\begin{equation}\label{adjoint'}
\begin{cases}
\varphi_t +\psi_x+\psi_{xxx}= 0, & \text{in} \,\, (0,L)\times (0,\tau),\\
\psi_t +\varphi_x +\varphi_{xxx}=0  & \text{in} \,\, (0,L)\times (0,\tau), \\
\varphi(0,t)=\varphi(L,t)=\varphi_{x}(0,t)=0,& \text{on} \,\,  (0,\tau),\\
\psi(0,t)=\psi(L,t)=\psi_{x}(L,t)=0, & \text{on} \,\,  (0,\tau),\\
\varphi(x,\tau)=\varphi_\tau,\quad \psi(x,\tau)=\psi_\tau, & \text{on} \,\, (0,L).
\end{cases}
\end{equation}
\end{definition}
The following result ensures the existence and uniqueness of the solution for the system \eqref{nhs1}.
\begin{lemma}\label{existentransposition1}
Let $T>0$, $(\eta_0,w_0) \in \overline{X}_2$, $(h_1,h_2) \in L^2(0,T;X_{-1})$ and $(f,g) \in [L^2(0,T)]^2$. There exists a unique solution $(\eta,w) \in C([0,T];\overline{X}_2)$ of the system \eqref{nhs1}. Moreover, there exists a positive constant $C_T$, such that 
\begin{equation}\label{depend1}
\|(\eta(\tau),w(\tau))\|_{\overline{X}_2}\leq C_T \left(\|(\eta_0,w_0)\|_{\overline{X}_2}+\|(f,g)\|_{[L^{2}(0,T)]^2}+ \|(h_1,h_2)\|_{L^2(0,T;X_{-1})}\right), 
\end{equation}
for all $\tau \in [0,T]$.
\end{lemma}

\begin{proof}
Let $T>0$ and $\tau\in[0,T]$. Note that making the change of variable $(x,t)\mapsto (\varphi(x,\tau-t),\psi(x,\tau-t))$,  from \eqref{new1} and the fact that the operator $A$ is skew adjoint, we have that the solution of \eqref{adjoint'}  is given by $$(\varphi,\psi)=S^*(\tau-t)(\varphi_\tau,\psi_\tau)=-S(\tau-t)(\varphi_\tau,\psi_\tau).$$ Moreover,  $$(\varphi,\psi)\in C(\R;\overline{X}_2),$$ where $\{S(t)\}_{t\geq0}$ is a $C_0$-group generated by $A$. In particular, there exists $C_T>0$, such that
\begin{equation}\label{adj1'}
\|(\varphi(t),\psi(t))\|_{\overline{X}_2}=\|S^*(\tau-t)(\varphi_\tau,\psi_\tau)\|_{\overline{X}_2}\leq C_T\|(\varphi_\tau,\psi_{\tau})\|_{\overline{X}_2}, \quad \forall t\in [0,\tau].
\end{equation}
Let us define $L$ as a linear functional given by the right hand side of \eqref{transp1}, that is,
\begin{multline*}
L(\varphi_\tau,\psi_\tau)= \left( (\eta_0,w_0),S^*(\tau)(\varphi_\tau,\psi_\tau) \right)_{\overline{X}_2}  +\int_0^\tau \left((g(t),f(t)),\frac{d}{dx}(S^*(\tau-t)(\varphi_\tau,\psi_\tau))\Big|^L_0\right)dt \\
+ \int_0^\tau \left\langle(h_1(t), h_2(t)), S^*(\tau-t)(\varphi_\tau,\psi_\tau)\right\rangle_{(H^{-1}\times H^{1}_0)^2}dt.
\end{multline*}
\vspace{0.1cm}
\noindent \textit{Claim.} $L$ belongs to  $\LL(\overline{X}_2;\mathbb{R})$. 
\vspace{0.1cm}

Indeed, 
\begin{align*}
|L(\varphi_\tau,\psi_\tau)|\leq  & C_T\|(\eta_0,w_0)\|_{\overline{X}_2}\|(\varphi_\tau,\psi_\tau))\|_{\overline{X}_2} + C_T\|(\varphi_\tau,\psi_\tau)\|_{\overline{X}_2}\|(h_1,h_2)\|_{L^1(0,T;X_{-1})} \\
&+ \|(f,g)\|_{[L^2(0,T)]^2}\left(\|\psi_x(0,\cdot)\|_{L^2(0,T)} +\|\varphi_x(L,\cdot)\|_{L^2(0,T)}\right) \\
\leq  & C_T\left(\|(\eta_0,w_0)\|_{\overline{X}_2}+ \|(h_1,h_2)\|_{L^2(0,T;X_{-1})}\right)\|(\varphi_\tau,\psi_\tau))\|_{\overline{X}_2} \\
&+ \|(f,g)\|_{[L^2(0,T)]^2}\|(\varphi_x,\psi_x)\|_{L^2(0,T;[L ^{\infty}(0,L)]^2)}  \\
\leq  & C_T\left(\|(\eta_0,w_0)\|_{\overline{X}_2}+ \|(h_1,h_2)\|_{L^2(0,T;X_{-1})}\right)\|(\varphi_\tau,\psi_\tau))\|_{\overline{X}_2} \\
&+ \|(f,g)\|_{X_0}\|(\varphi,\psi)\|_{C([0,T];\overline{X}_2)}  \\
\leq  & C_T\left(\|(\eta_0,w_0)\|_{\overline{X}_2} + \|(f,g)\|_{X_0} + \|(h_1,h_2)\|_{L^2(0,T;X_{-1})}\right)\|(\varphi_\tau,\psi_\tau))\|_{\overline{X}_2}, 
\end{align*}
where in the last inequality we use \eqref{adj1'}. Then, from the Riesz representation Theorem, there exist one and only one $(\eta_\tau, w_\tau) \in  \overline{X}_2$ such that
\begin{equation}\label{eqn1}
\left( (\eta_\tau,w_\tau),(\varphi_\tau,\psi_\tau)\right)_{\overline{X}_2}=L(\varphi_\tau,\psi_\tau), \quad \text{with} \quad  \|(\eta_\tau,w_\tau)\|_{\overline{X}_2}=\|L\|_ {\LL(\overline{X}_2;\mathbb{R})}
\end{equation}
and the uniqueness of the solution to the problem \eqref{nhs1} holds.

We prove now that the solution of the system \eqref{nhs1} satisfies  \eqref{depend1}. Let $(\eta,w) : [0, T] \longrightarrow \overline{X}_2$ be defined by
\begin{equation}\label{eqn2}
(\eta(\tau),w(\tau)):=(\eta_\tau,w_\tau), \quad \forall \tau \in [0,T].
\end{equation}
From \eqref{eqn1} and \eqref{eqn2}, \eqref{transp1} follows and 
\begin{align*}
\|(\eta(\tau),w(\tau))\|_{\overline{X}_2}= \|L\|_{ \LL(\overline{X}_2;\mathbb{R})} \leq C_T\left(\|(\eta_0,w_0)\|_{\overline{X}_2} + \|(f,g)\|_{X_0} + \|(h_1,h_2)\|_{L^2(0,T;\overline{X}_{-1})}\right).
\end{align*}
In order to prove that the solution  $(\eta, w)$ belongs to $ C([0,T],\overline{X}_2)$, let $\tau \in [0,T]$ and $\{\tau_n\}_{n\in \N}$ be a sequence such that 
\begin{equation}\label{eqn3}
\tau_n \longrightarrow \tau, \quad \text{as $n\rightarrow \infty$.}
\end{equation}
Consider $(\varphi_\tau,\psi_\tau) \in \overline{X}_2$ and $\{(\varphi_{\tau_n},\psi_{\tau_n})\}_{n\in \N}$ be a sequence in $\overline{X}_2$ such that
\begin{equation}\label{eqn4}
(\varphi_{\tau_n},\psi_{\tau_n}) \rightarrow  (\varphi_{\tau},\psi_\tau) \quad \text{strong in $\overline{X}_2$, as $n \rightarrow \infty$.}
\end{equation}
Moreover, note that 
\begin{align}
\lim_{n\rightarrow \infty}\left( (\eta_0,w_0), S^*(\tau_n)(\varphi_{\tau_n},\psi_{\tau_n})\right)_{\overline{X}_2}=  \left((\eta_0,w_0),S^*(\tau) (\varphi_{\tau},\psi_{\tau})\right)_{\overline{X}_2}. \label{eqn5}
\end{align}
Indeed, 
\begin{multline*}
\lim_{n\rightarrow \infty}\left( (\eta_0,w_0), S^*(\tau_n)(\varphi_{\tau_n},\psi_{\tau_n})\right)_{\overline{X}_2}=  \lim_{n\rightarrow \infty}\left( (\eta_0,w_0), S^*(\tau_n)\left( (\varphi_{\tau_n},\psi_{\tau_n})-(\varphi_{\tau},\psi_{\tau})\right) \right)_{\overline{X}_2} \\
 + \lim_{n\rightarrow \infty}\left( (\eta_0,w_0), S^*(\tau_n)(\varphi_{\tau},\psi_{\tau})\right)_{\overline{X}_2}.
\end{multline*}
From \eqref{eqn4} and since $\{S(t)\}_{t \geq0}$ is a strongly continuous group of continuous linear operators on $X_0$, we have 
\begin{align*}
\lim_{n\rightarrow \infty}\left( (\eta_0,w_0), S^*(\tau_n)\left( (\varphi_{\tau_n},\psi_{\tau_n})-(\varphi_{\tau},\psi_{\tau})\right) \right)_{\overline{X}_2}&=0
\end{align*}
and consequently, 
\begin{align*}
\lim_{n\rightarrow \infty}\left( (\eta_0,w_0), S^*(\tau_n)(\varphi_{\tau},\psi_{\tau})\right)_{\overline{X}_2}&=\left( (\eta_0,w_0), S^*(\tau)(\varphi_{\tau},\psi_{\tau})\right)_{\overline{X}_2}.
\end{align*}
Thus, \eqref{eqn5} follows. Now, 
extending by zero the functions $f$, $g$ and $h_i$, for $i=1,2$, to obtain elements of $L^2(-T,T)$ and $L^2(-T,T;\overline{X}_{-1})$, that is, $$f\equiv g\equiv0 \quad  \text{ on } (-T,0) \quad  \text{and} \quad h_i \equiv 0\quad  \text{a.e in  $(-T,0)\times(0,L)$} $$
and setting $s=\tau_n-t$, we have that
\begin{multline}\label{eqn6}
\int_0^{\tau_n}\left((g(t),f(t)),\frac{d}{dx}(S^*(\tau_n-t)(\varphi_{\tau_n},\psi_{\tau_n}))\Big|_0^L\right)dt \\
=\int_0^T \chi_n(s)\left((g(\tau_n-s),f(\tau_n-s)),\frac{d}{dx}(S^*(s)(\varphi_{\tau_n},\psi_{\tau_n}))\Big|_0^L\right) dt
\end{multline}
and
\begin{multline}\label{eqn6'}
\int_0^{\tau_n}\left\langle(h_1(t),h_2(t)),S^*(\tau_n-t)(\varphi_{\tau_n},\psi_{\tau_n})\right\rangle_{H^{-1}\times H^1_0}dt \\
=\int_0^T \chi_n(s)\left\langle(h_1(\tau_n-s),h_2(\tau_n-s)),S^*(s)(\varphi_{\tau_n},\psi_{\tau_n})\right\rangle_{H^{-1}\times H^1_0} dt
\end{multline}
where $\chi_n$ is the characteristic functions of $[0,\tau_n]$. Similarly, doing $s=\tau-t$, we get
\begin{multline}\label{eqn7}
\int_0^\tau\left((g(t),f(t)), \frac{d}{dx}(S^*(\tau-t)(\varphi_\tau,\psi_\tau))\Big|_0^L\right)dt \\
=\int_0^T \chi(s)\left((g(\tau-s),f(\tau-s)),\frac{d}{dx}(S^*(s)(\varphi_\tau,\psi_\tau))\Big|_0^L\right)dt
\end{multline}
and
\begin{multline}\label{eqn7'}
\int_0^\tau\left\langle (h_1(t),h_2(t)), S^*(\tau-t)(\varphi_\tau,\psi_\tau) \right\rangle_{H^{-1}\times H^1_0}dt \\
=\int_0^T \chi(s)\left\langle ( h_1(\tau-s),h_2(\tau-s)), S^*(s)(\varphi_\tau,\psi_\tau) \right\rangle_{H^{-1}\times H^1_0}dt, 
\end{multline}
where$\chi$ is the characteristic functions of $[0,\tau]$. Thus, by using the convergence dominated Theorem, one gets that 
\begin{equation}\label{eqn8}
\chi_n(\cdot)(g(\tau_n-\cdot),f(\tau_n-\cdot)) \longrightarrow \chi(\cdot)(g(\tau-\cdot),f(\tau-\cdot)) \quad \text{in $L^2((0,T);\R^2)$ as $n\rightarrow\infty$}.
\end{equation} 
Moreover, since the translation in time is continuous in $L^2(0,T, X_{-1})$ and using again convergence dominated theorem, we obtain
\begin{equation}\label{eqn8'}
\chi_n(\cdot)(h_1(\tau_n-\cdot,\cdot),h_2(\tau_n-\cdot,\cdot)) \longrightarrow \chi(\cdot)(h_1(\tau-\cdot,\cdot),h_2(\tau-\cdot,\cdot)) \quad \text{in $L^2(0,T, X_{-1})$ as $n\rightarrow\infty$}.
\end{equation} 
Observe that, by the group properties of $S^*$, we have that
\begin{align*}
\int_0^T\left| \frac{d}{dx}(S^*(t)(\varphi_{\tau},\psi_{\tau}))\Big|_0^L\right|^2dt &= \int_0^T \left(|\varphi_x(L,\tau-t)|^2+|\psi_x(0,\tau-t)|^2\right)dt  \\
&\leq \|(\varphi_x,\psi_x)\|_{L^2(-T,T,[L^{\infty}(0,L)]^2)} \\
&\leq C\|(\varphi,\psi)\|_{C([-T,T];\overline{X}_2)}^2 \\
&\leq C \|(\varphi_\tau,\psi_\tau)\|^2_{\overline{X}_2}, \quad \forall (\varphi_\tau,\psi_\tau) \in \overline{X}_2.
\end{align*}
Thus, the linear map $(\varphi_\tau,\psi_\tau) \in \overline{X}_2 \mapsto \frac{d}{dx}(S^*(\cdot)(\psi_{\tau},\varphi_{\tau}))\Big|_0^L$ belongs to $L^2(0,T; \R^2)$ and it is continuous. Since a continuous linear map between two Hilbert spaces is weakly continuous, \eqref{eqn4} implies that 
\begin{align}\label{eqn9}
 \frac{d}{dx}(S^*(\cdot)(\varphi_{\tau_n},\psi_{\tau_n}))\Big|_0^L \rightharpoonup  \frac{d}{dx}(S^*(\cdot)(\varphi_\tau,\psi_\tau))\Big|_0^L \quad \text{weakly in } L^2([-T,T];\R^2) \text{ as } n\rightarrow\infty.
\end{align}
Similarly, by the strong continuity of the group, it follows that
\begin{align}\label{eqn9'}
S^*(\cdot)(\varphi_{\tau_n},\psi_{\tau_n}) \rightharpoonup   S^*(\cdot)(\varphi_\tau,\psi_\tau)  \quad \text{weakly in } L^2(-T,T,X_0) \text{ as } n\rightarrow\infty,
\end{align}
in particular, we obtain that 
\begin{align}\label{eqn9''}
S^*(\cdot)(\varphi_{\tau_n},\psi_{\tau_n}) \rightharpoonup   S^*(\cdot)(\varphi_\tau,\psi_\tau)  \quad \text{weakly in }  L^2(-T,T,\overline{X}_1)  \text{ as } n\rightarrow\infty.
\end{align}
Thus, \eqref{eqn6}-\eqref{eqn9''} yields that 
\begin{multline}\label{eqn10}
\lim_{n\rightarrow\infty}\int_0^{\tau_n} \left((g(t),f(t)),\frac{d}{dx}(S^*(\tau_n-t)(\varphi_{\tau_n},\psi_{\tau_n}))\Big|_0^L\right) dt \\
=\int_0^\tau \left((g(t),f(t)),\frac{d}{dx}(S^*(\tau-t)(\varphi_{\tau},\psi_{\tau}))\Big|_0^L\right) dt 
\end{multline}
and
\begin{multline}\label{eqn10'}
\lim_{n\rightarrow\infty}\int_0^{\tau_n} \left\langle (h_1(t),h_2(t)),S^*(\tau_n-t)(\varphi_{\tau_n},\psi_{\tau_n})\right\rangle_{H^{-1}\times H^1_0} dt \\
=\int_0^\tau \left\langle(h_1(t),h_2(t)), S^*(\tau-t)(\varphi_{\tau},\psi_{\tau})\right\rangle_{H^{-1}\times H^1_0}dt .
\end{multline}
Finally, from \eqref{eqn1}, \eqref{eqn2}, \eqref{eqn5}, \eqref{eqn10} and \eqref{eqn10'}, one gets
\begin{align*}
\left( (\eta(\tau_n),w(\tau_n)),(\varphi_{\tau_n},\psi_{\tau_n})\right)_{\overline{X}_2 }\longrightarrow \left( (\eta(\tau),w(\tau)),(\varphi_{\tau},\psi_{\tau}))\right)_{\overline{X}_2} \quad \text{as $n\rightarrow\infty,$}
\end{align*}
which implies that 
\begin{align*}
\left(\eta(\tau_n),w(\tau_n)\right) \longrightarrow \left( \eta(\tau),w(\tau)\right) \quad \text{in $\overline{X}_2$, as $n\rightarrow\infty$}.
\end{align*}
This concludes the proof.
\end{proof}
The next result establish the well-posedness of the non-homogeneous feedback linear system associated to \eqref{nhs1}. 
\begin{lemma}\label{exitencefeedback1}
Let $T>0$ and $F_i:X_0 \longrightarrow \R$ be a continuous bilinear map, for $i=1,2$. Then, for every $(\eta_0,w_0)$ in $\overline{X}_2$ and $(h_1,h_2) \in L^2(0,T;X_{-1})$, there exists a unique solution $(\eta, w)$ of the system \eqref{nhs1} such that  $$(\eta, w)\in C([0,T];\overline{X}_2),$$  with $f(t):=F_1(\eta(t),w(t))$ and $g(t):=F_2(\eta(t),w(t))$. Moreover, for some positive constant $C=C(T)$, we have 

\begin{align*}
\|(\eta(t),w(t))\|_{\overline{X}_2} \leq C\left( \|(\eta_0,w_0)\|_{\overline{X}_2} + \|(h_1,h_2)\|_{L^2(0,T;X_{-1})}\right), \quad \forall t \in [0,T].
\end{align*}
\end{lemma}
\begin{proof}
Firstly, note that if $(\eta,w) \in C([0,T];\overline{X}_2)$, we have that 
\begin{equation*}
\|F_i(\eta,w)\|_{L^2(0,T)}^2=\int_0^T|F_i(\eta(\cdot, t),w(\cdot, t))|^2dt \leq CT\|F_i\|_{\LL(X_0;\mathbb{R})}^2 \|(\eta,w)\|^2_{C([0,T];\overline{X}_2)},
\end{equation*}
 then $F_i(\eta(\cdot,t),w(\cdot,t)) \in L^{2}(0,T)$, for $i=1,2$. 
 
 Let $0<\beta \leq T$ that will be determinate later. For each $(\eta_0,w_0) \in \overline{X}_2$, consider the map 
\begin{align*}
\begin{array}{l c l c}
\Gamma : & C([0,\beta];\overline{X}_2) & \longrightarrow & C([0,\beta];\overline{X}_2) \\
         & (\eta,w) & \longmapsto & \Gamma(\eta,w)=(u,v)   
\end{array}
\end{align*}
where, $(u,v)$ is the solution of the system \eqref{nhs1} with $f(t)=F_1(\eta(t),w(t))$ and $ g(t)=F_2(\eta(t),w(t))$. By Lemma \ref{existentransposition1}, the linear map $\Gamma$ is well defined. Furthermore, there exists a positive constant $C$, such that
\begin{align*}
\|\Gamma(\eta,w)\|_{C([0,\beta];\overline{X}_2)} \leq C_{\beta} (\|(\eta_0,w_0)\|_{\overline{X}_2}+\|(F_1(\eta,w),F_2(\eta,w))\|_{[L^2(0,\beta)]^2} + \|(h_1,h_2)\|_{L^2(0,T;X_{-1})} ).
\end{align*}
Then, 
\begin{multline*}
\|\Gamma(\eta,w)\|_{C([0,\beta];\overline{X}_2)} \leq C_T \left( \|(\eta_0,w_0)\|_{\overline{X}_2}+ \|(h_1,h_2)\|_{L^2(0,T;X_{-1})} \right) \\
+ C_T\beta^{\frac12} \left(\|F_1\|_{\LL(X_0;\mathbb{R})}^2+\|F_2\|_{\LL(X_0;\mathbb{R})}^2\right) \|(\eta,w)\|_{C([0,\beta];\overline{X}_2)}.
\end{multline*}
Let us define $$B_{R}(0):=\{ (\eta,w)\in C([0,\beta];\overline{X}_2); \|(\eta,w)\|_{C([0,\beta];\overline{X}_2)} \leq R\},$$ with $$R=2C_T\left( \|(\eta_0,w_0)\|_{\overline{X}_2}+ \|(h_1,h_2)\|_{L^2(0,T;X_{-1})} \right).$$ Choosing $\beta$ such that 
\begin{align*}
C_T\beta^{\frac12}\left\lbrace\|F_1\|_{\LL(X_0;\mathbb{R})}^2+\|F_2\|_{\LL(X_0;\mathbb{R})}^2\right\rbrace\leq \frac12,
\end{align*}
it implies that $\|\Gamma(\eta,w)\|_{C([0,\beta];\overline{X}_2)} \leq R$, for all $(\eta,w) \in B_{R}(0)$, i.e,  $\Gamma$ maps $B_R(0)$ into $B_R(0)$. Furthermore, note that 
\begin{align*}
\|\Gamma(\eta_1,w_1)-\Gamma(\eta_2,w_2)\|_{C([0,\beta];\overline{X}_2)} 
\leq & C_T\beta^{\frac12}\left(\|F_1\|_{\LL(X_0;\mathbb{R})}^2+\|F_2\|_{\LL(X_0;\mathbb{R})}^2\right)\|(\eta_1-\eta_2,w_1-w_2)\|_{C([0,\beta];\overline{X}_2)} \\
\leq & \frac12 \|(\eta_1-\eta_2,w_1-w_2)\|_{C([0,\beta];\overline{X}_2)}.
\end{align*}
Hence, $\Gamma: B_R(0) \longrightarrow B_R(0)$ is a contraction and, by Banach fixed point theorem, we obtain a unique $(\eta, w) \in  B_R(0)$, such that $\Gamma(\eta,w)=(\eta,w)$ and 
\begin{align*}
\|(\eta,w)\|_{C([0,\beta];\overline{X}_2)} \leq 2C_T\left( \|(\eta_0,w_0)\|_{\overline{X}_2}+ \|(h_1,h_2)\|_{L^2(0,T;X_{-1})} \right).
\end{align*}
Since the choice of  $\beta$ is independent of $(\eta_0, w_0)$, the standard continuation extension argument yields that the solution $(\eta,w)$ belongs to $C([0,\beta];\overline{X}_2)$, thus, the proof is complete.
\end{proof}

Now, we are able to prove the main result of this subsection.
\begin{proof}[\textbf{Proof of Theorem \ref{nonlinearexitence1}}]
Let $T>0$ and $\|(\eta_0,w_0)\|_{\overline{X}_2}<\rho$, where $\rho>0$ will be determined later. Note that for $(\eta,w) \in C([0,T];\overline{X}_2)$, there exists a positive constant $C_1$ such that 
\begin{align}
\|\eta w_x\|^2_{L^2(0,T; L^2(0,L))} &\leq \int_0^T \|\eta(t)\|^2_{L^\infty(0,L)} \|w_x(t)\|^2_{L^2(0,L)}dt\nonumber\\& \leq C_1' \int_0^T \|\eta(t)\|^2_{H^{2}(0,L)} \|w(t)\|^2_{H^{2}(0,L)}dt\label{eqn11}\\
& \leq C_1 T \|(\eta,w)\|^4_{C([0,T]; \overline{X}_2)}.\nonumber
\end{align}
This implies that for any $(\eta,w) \in C([0,T]; \overline{X}_2)$, we have that $$((\eta w)_x, ww_x)\in L^2(0,T;X_0) \subset L^2(0,T;X_{-1}). $$  Consider the following linear map 
\begin{align*}
\begin{array}{c c c c}
\Gamma : & C([0,T]; \overline{X}_2) & \longrightarrow & C([0,T]; \overline{X}_2)\\
         & (\eta,w) & \longmapsto & \Gamma(\eta,w)=(\overline{\eta},\overline{w}),   
\end{array}
\end{align*}
where $(\overline{\eta},\overline{w})$ is the solution of the system \eqref{nhs1} with $(h_1,h_2):=(-(\eta w)_x,- ww_x)$ in $L^2(0,T;X_{-1})$, $f(t)=F_1(\overline{\eta}(t),\overline{w}(t))$ and $g(t)=F_2(\overline{\eta}(t),\overline{w}(t))$. 

\vspace{0.2cm}
\noindent \textit{Claim.} The map $\Gamma$ is well-defined, maps $B_R(0)$ into it self and is a contraction in a ball. 
\vspace{0.2cm}

Indeed, firstly note that Lemma \eqref{exitencefeedback1} ensures that $\Gamma$ is well-defined, moreover using Lemma \ref{existentransposition1}, there exists a positive constant $C_T$, such that
\begin{align*}
\|\Gamma(\eta,w)\|_{ C([0,T]; \overline{X}_2} \leq C_{T} (\|(\eta_0,w_0)\|_{\overline{X}_2}+\|(\eta w)_x\|_{L^2(0,T; L^2(0,L))}+ \|w w_x\|_{L^2(0,T; L^2(0,L))}). 
\end{align*}
Then, \eqref{eqn11} yields that
\begin{align}\label{Gamma}
\|\Gamma(\eta,w)\|_{ C([0,T]; \overline{X}_2} \leq C_{T} \|(\eta_0,w_0)\|_{\overline{X}_2} + 3T^{1/2}C_TC_1 \|(\eta,w)\|^2_{C([0,T];\overline{X}_2)}. 
\end{align}
Consider the ball  $B_R(0) = \left\lbrace(\eta,w) \in C([0,T]; \overline{X}_2 ): \|(\eta,w)\|_{C([0,T]; \overline{X}_2)}\leq R\right\rbrace$, where $$R= 2C_T\|(\eta_0,w_0)\|_{\overline{X}_2}.$$ From the estimate \eqref{Gamma} we get that
\begin{align*}
\|\Gamma(\eta,w)\|_{C([0,T];\overline{X}_2)} &\leq \frac{R}{2}+ 3T^{1/2}C_TC_1R^2 <\frac{R}{2}+ 6T^{1/2}C_T^2C_1\rho R, \quad \forall (\eta,w) \in B_R(0).
\end{align*}
Consequently, if we choose $\rho>0$ such that $24T^{1/2}C_T^2C_1\rho <1$, $\Gamma$ maps the ball $B_R(0)$ into itself. Finally, note that 
\begin{align*}
\|\Gamma&(\eta_1,w_1)-\Gamma(\eta_2,w_2)\|_{{C([0,T];\overline{X}_2)}}  \leq C_T \|\left( (\eta_1w_1)_x-(\eta_2w_2)_x, w_1w_{1,x}-w_2w_{2,x}\right)\|_{L^1(0,T;X_0)}\\
\leq &T^{1/2}C_TC_1((\|w_1\|_{C([0,T];H^{2}(0,L))}+ \|w_2\|_{C([0,T];H^{2}(0,L))} )\|\eta_1-\eta_2\|_{C([0,T];H^{2}(0,L))} \\
& \quad +T^{1/2}C_TC_1((\|\eta_1\|_{C([0,T];H^{2}(0,L))}+ \|\eta_2\|_{C([0,T];H^{2}(0,L))}+\|w_1\|_{C([0,T];H^{2}(0,L)) }  \\
& \qquad \qquad \qquad \qquad \qquad \qquad \qquad \qquad  +\|w_2\|_{C([0,T];H^{2}(0,L)) } )\|w_1-w_2\|_{C([0,T];H^{2}(0,L))} \\
\leq &12T^{1/2}C_T^2C_1\rho\|(\eta_1-\eta_2,w_1-w_2)\|_{C([0,T];\overline{X}_2)}. 
\end{align*}
Therefore, 
\begin{align*}
\|\Gamma(\eta_1,w_1)-\Gamma(\eta_2,w_2)&\|_{C([0,T];\overline{X}_2)} \leq \frac{1}{2}\|(\eta_1-\eta_2,w_1-w_2)\|_{{C([0,T];\overline{X}_2)}}, \quad \forall (\eta,w) \in B_T(0).
\end{align*}
Hence, $\Gamma: B_R(0) \longrightarrow B_R(0)$ is a contraction and the claim is archived. 

Thanks to Banach fixed point theorem, we obtain a unique $(\eta, w) \in  B_R$, such that $\Gamma(\eta,w)=(\eta,w)$ and 
\begin{align*}
\|(\eta,w)\|_{C([0,T];\overline{X}_2)} \leq 2C_{T} \|(\eta_0,w_0)\|_{\overline{X}_2}.
\end{align*}
Thus, the proof is complete.
\end{proof}


\section{Well-posedness for the stationary system }\label{st_p}
In this section we are interested to show the well-posedness of the stationary system \eqref{esta1'}-\eqref{esta1.1'}. For a better understanding, we recall the definitions of the operators $K$ and $S$:
\begin{equation}\label{deftransf}
(Kv)(x):=\int_0^L k(x,y)v(y)dy \quad \text{and} \quad (Sv)(x):=\int_0^L s(x,y)v(y)dy, \quad \text{for all $v \in L^2(0,L)$,}
\end{equation}
where $(k,s)$ is the solution of stationary problem
\begin{equation}\label{esta1}
\begin{cases}
k_{yyy}+k_y+k_{xxx}+k_x+\lambda s=0, & \text{in $(0,L)\times (0,L)$}, \\
s_{yyy}+s_y+s_{xxx}+s_x+\lambda k=\lambda \delta(x-y), & \text{in $(0,L)\times (0,L)$},
\end{cases}
\end{equation}
with boundary condition 
\begin{equation}\label{esta1.1}
\left\lbrace\begin{array}{l l }
k(x,0)=k(x,L)=k(0,y)=k(L,y)=k_y(x,0)=k_y(x,L)=0, &   \text{on $(0,L)$}, \\
s(x,0)=s(x,L)=s(0,y)=s(L,y)=s_y(x,0)=s_y(x,L)=0, &   \text{on $(0,L)$}, 
\end{array}\right.
\end{equation}
where $\lambda \in \R\setminus \{0\}$ and $\delta(x-y)$ denotes the Dirac measure on the diagonal of the square $[0,L]\times [0,L]$.


\subsection{Well-posedness: Stationary system}\label{sec21}
In this subsection we study the well-posedness of the system \eqref{esta1}-\eqref{esta1.1}. In this direction, we will define the solution by transposition. 

Let us consider the set 
\begin{equation*}
\EE := \left\lbrace \rho \in C^{\infty}([0,L]\times[0,L]): \rho(0,y)= \rho(L,y)=\rho(x,0)=\rho(x,L)=\rho_x(0,y)=\rho_x(L,y)=0\right\rbrace
\end{equation*}
and $\GG$ be the set given by
\begin{align*}
\GG:= \left\lbrace k \in H_0^1((0,L)\times(0,L)) : \begin{array}{c}
(x \in (0,L) \mapsto k_x(x,\cdot) \in  L^2(0,L)) \in C([0,L];L^2(0,L)), \\
(y \in (0,L) \mapsto k_y(\cdot,y) \in  L^2(0,L)) \in C([0,L];L^2(0,L)), \\
k_y(\cdot,0)=k_y(\cdot,L)=0.
\end{array} \right\rbrace.
\end{align*}
We call $(k,s) \in \GG \times \GG$ a solution to \eqref{esta1}-\eqref{esta1.1} if 
\begin{multline}\label{e1rapid}
\int_0^L\int_0^L \left[ \rho_{yyy}(x,y)+\rho_y(x,y)+\rho_{xxx}(x,y)+\rho_x(x,y)\right]k(x,y)dxdy  \\ 
-\lambda \int_0^L\int_0^L\rho(x,y)s(x,y)dxdy=0, \quad \text{for every $\rho \in \EE$}
\end{multline}
and 
\begin{multline}\label{e2rapid}
\int_0^L\int_0^L \left[ \rho_{yyy}(x,y)+\rho_y(x,y)+\rho_{xxx}(x,y)+\rho_x(x,y)\right]s(x,y)dxdy  \\
-\lambda \int_0^L\int_0^L\rho(x,y)k(x,y)dxdy +\lambda \int_0^L\rho(x,x)dx=0 \quad \text{for every $\rho \in \EE$}.
\end{multline}

We can enunciate the well-posedness result as follows.
\begin{thm}\label{mod_coron}
For any $\lambda \in \R\setminus \{0\}$, system \eqref{esta1}-\eqref{esta1.1} has a unique solution in $\GG \times \GG$.
\end{thm}
\begin{proof} Consider $v\in\GG$ solution of the equation
\begin{equation}\label{coron1}
\begin{cases}
\overline{k}_{yyy}+\overline{k}_y+\overline{k}_{xxx}+\overline{k}_x+\lambda \overline{k}= \lambda \delta (x-y) & \text{in $(0,L)\times(0,L)$}, \\
\overline{k}(x,0)=\overline{k}(x,L)=0, & \text{on $(0,L)$}, \\
\overline{k}_y(x,0)=\overline{k}_y(x,L)=0, & \text{on $(0,L)$}, \\
\overline{k}(0,y)=\overline{k}(L,y)=0, & \text{on $(0,L)$}
\end{cases}
\end{equation}
and $u\in\GG$ solution of the equation
\begin{equation}\label{coron1'}
\begin{cases}
\overline{k}_{yyy}+\overline{k}_y+\overline{k}_{xxx}+\overline{k}_x-\lambda \overline{k}= -\lambda \delta (x-y) & \text{in $(0,L)\times(0,L)$}, \\
\overline{k}(x,0)=\overline{k}(x,L)=0, & \text{on $(0,L)$}, \\
\overline{k}_y(x,0)=\overline{k}_y(x,L)=0, & \text{on $(0,L)$}, \\
\overline{k}(0,y)=\overline{k}(L,y)=0, & \text{on $(0,L)$}.
\end{cases}
\end{equation}
The existence of such $v$ and $u$ in $\GG$ is provided by \cite[Lemma 2.1]{coron2014}.  Then, in this case, we easily see that, if we define  
\begin{equation*}
k:=\frac{v+u}{2}\in \GG \quad \text{and} \quad s:=\frac{v-u}{2}\in\GG,
\end{equation*}
then $k$ and $s$ satisfies the boundary conditions \eqref{esta1.1}. 

Now we prove that $(k,s)\in\GG\times\GG$ is solution of the system \eqref{esta1}-\eqref{esta1.1} in the sense given in the beginning of this subsection. Let $\rho \in \EE$. Since $v$ and $u$ are solutions of \eqref{coron1} and \eqref{coron1'}, respectively, it follows that 
\begin{equation*}
\int_0^L\int_0^L \left[ \rho_{yyy}(x,y)+\rho_y(x,y)+\rho_{xxx}(x,y)+\rho_x(x,y)-\lambda\rho(x,y)\right]v(x,y)dxdy  +\lambda \int_0^L\rho(x,x)dx=0
\end{equation*}
and 
\begin{equation*}
\int_0^L\int_0^L \left[ \rho_{yyy}(x,y)+\rho_y(x,y)+\rho_{xxx}(x,y)+\rho_x(x,y)+\lambda\rho(x,y)\right]u(x,y)dxdy  -\lambda \int_0^L\rho(x,x)dx=0.
\end{equation*}
Thus,  integrating by parts, we obtain
\begin{align*}
\int_0^L\int_0^L&\left\lbrace \left[ \rho_{yyy}(x,y)+\rho_y(x,y)+\rho_{xxx}(x,y)+\rho_x(x,y)\right]k(x,y)  -\lambda \rho(x,y)s(x,y)\right\rbrace dxdy \\ 
=&\frac12 \int_0^L\int_0^L \left[ \rho_{yyy}(x,y)+\rho_y(x,y)+\rho_{xxx}(x,y)+\rho_x(x,y)\right](v(x,y)+u(x,y))dxdy \\
&-\frac12\lambda \int_0^L\int_0^L\rho(x,y)(v(x,y)-u(x,y))dxdy 
\\
=& \frac12 \int_0^L\int_0^L \left[ \rho_{yyy}(x,y)+\rho_y(x,y)+\rho_{xxx}(x,y)+\rho_x(x,y) -\lambda \rho(x,x)\right]v(x,y)dxdy \\
&+ \frac12 \int_0^L\int_0^L \left[ \rho_{yyy}(x,y)+\rho_y(x,y)+\rho_{xxx}(x,y)+\rho_x(x,y) +\lambda \rho(x,x)\right]u(x,y)dxdy \\
=& 0
\end{align*}
and 
\begin{align*}
\int_0^L\int_0^L &\left\lbrace\left[ \rho_{yyy}(x,y)+\rho_y(x,y)+\rho_{xxx}(x,y)+\rho_x(x,y)\right]s(x,y) -\lambda \rho(x,y)k(x,y)\right\rbrace dxdy\\
=&\frac12\int_0^L\int_0^L \left[ \rho_{yyy}(x,y)+\rho_y(x,y)+\rho_{xxx}(x,y)+\rho_x(x,y)\right](v(x,y)-u(x,y))dxdy \\
&-\frac{\lambda}{2}\int_0^L\int_0^L \rho(x,y)(v(x,y)+u(x,y))dxdy\\
=&\frac12\int_0^L\int_0^L \left[ \rho_{yyy}(x,y)+\rho_y(x,y)+\rho_{xxx}(x,y)+\rho_x(x,y)-\lambda\rho(x,y)\right]v(x,y)dxdy \\
&-\frac12\int_0^L\int_0^L \left[ \rho_{yyy}(x,y)+\rho_y(x,y)+\rho_{xxx}(x,y)+\rho_x(x,y)+\lambda\rho(x,x)\right]u(x,y)dxdy \\
=&-\lambda \int_0^L\rho(x,x)dx.
\end{align*}
Consequently \eqref{e1rapid} and \eqref{e2rapid} are satisfied.

Now, we prove that there exists only one solution for the system \eqref{esta1}-\eqref{esta1.1}. In fact, suppose that $(k_1,s_1)$ and $(k_2,s_2)$ are solutions of the system \eqref{esta1}-\eqref{esta1.1} and consider $k:=k_1-k_2$ and $s:=s_1-s_2$. Then, $(k,s)$ is solution of 
\begin{equation*}
\begin{cases}
k_{yyy}+k_y+k_{xxx}+k_x+\lambda s=0, & \text{in $(0,L)\times (0,L)$}, \\
s_{yyy}+s_y+s_{xxx}+s_x+\lambda k=0, & \text{in $(0,L)\times (0,L)$}, 
\end{cases}
\end{equation*}
satisfying the boundary condition \eqref{esta1.1}. Note that $k+s$ solves the problem  
\begin{equation*}
\begin{cases}
\hat{k}_{yyy}+\hat{k}_y+\hat{k}_{xxx}+\hat{k}_x+\lambda \hat{k}= 0 & \text{in $(0,L)\times(0,L)$}, \\
\hat{k}(x,0)=\hat{k}(x,L)=0, & \text{on $(0,L)$}, \\
\hat{k}_y(x,0)=\hat{k}_y(x,L)=0, & \text{on $(0,L)$}, \\
\hat{k}(0,y)=\hat{k}(L,y)=0, & \text{on $(0,L)$}
\end{cases}
\end{equation*}
and $k-s$ solves 
\begin{equation*}
\begin{cases}
\hat{k}_{yyy}+\hat{k}_y+\hat{k}_{xxx}+\hat{k}_x-\lambda \hat{k}= 0 & \text{in $(0,L)\times(0,L)$}, \\
\hat{k}(x,0)=\hat{k}(x,L)=0, & \text{on $(0,L)$}, \\
\hat{k}_y(x,0)=\hat{k}_y(x,L)=0, & \text{on $(0,L)$}, \\
\hat{k}(0,y)=\hat{k}(L,y)=0, & \text{on $(0,L)$}.
\end{cases}
\end{equation*}
By uniqueness given in \cite[Lemma 2.1]{coron2014}, we have that $k+s=k-s=0$. Then, $k_1=k_2$ and $s_1=s_2$, and the proof is archived.
\end{proof}

\begin{remark}\label{remarkinverse}
Let $(k,s)$ be solution of the system \eqref{esta1}-\eqref{esta1.1}. Thus, it is clear that $k+s$ is solution of the system 
\begin{equation*}
\begin{cases}
\overline{k}_{yyy}+\overline{k}_y+\overline{k}_{xxx}+\overline{k}_x+\lambda \overline{k}= \lambda \delta (x-y) & \text{in $(0,L)\times(0,L)$}, \\
\overline{k}(x,0)=\overline{k}(x,L)=0, & \text{on $(0,L)$}, \\
\overline{k}_y(x,0)=\overline{k}_y(x,L)=0, & \text{on $(0,L)$}, \\
\overline{k}(0,y)=\overline{k}(L,y)=0, & \text{on $(0,L)$}.
\end{cases}
\end{equation*}
Then, from \cite[Lemma 3.1]{coron2014}, we have that $(I- (K+S))^{-1}$ exists and it is a continuous operator in $L^2(0,L)$. Similarly, we obtain that $(I- (K-S))^{-1}$ also belongs to $\LL(L^2(0,L))$, since $k-s$ is solution of the following system
\begin{equation*}
\begin{cases}
\overline{k}_{yyy}+\overline{k}_y+\overline{k}_{xxx}+\overline{k} -\lambda \overline{k}=- \lambda \delta (x-y) & \text{in $(0,L)\times(0,L)$}, \\
\overline{k}(x,0)=\overline{k}(x,L)=0, & \text{on $(0,L)$}, \\
\overline{k}_y(x,0)=\overline{k}_y(x,L)=0, & \text{on $(0,L)$}, \\
\overline{k}(0,y)=\overline{k}(L,y)=0, & \text{on $(0,L)$}.
\end{cases}
\end{equation*}
\end{remark}


\section{Rapid stabilization for the nonlinear system}\label{rs_p}
In this section we will establish the local rapid exponential stabilization for the system \eqref{n1'''}-\eqref{n1.1''} with the feedback laws
$f(t):=F_1(\eta(\cdot,t),w(\cdot,t))$ and $g(t):= F_2(\eta(\cdot,t),w(\cdot,t))$, with $F_i$ belongs to $\LL(X_0;\mathbb{R})$, $i=1,2$, defined by 
\begin{equation}\label{feedback1}
\begin{cases}
F_1(u,v) =   \displaystyle\int_0^L \left[ k_x(0,y)u(y)+s_x(0,y)v(y)\right]dy, & \forall (u,v) \in X_0, \\
F_2(u,v) = \displaystyle\int_0^L \left[ k_x(L,y)v(y)+s_x(L,y)u(y)\right]dy, &\forall (u,v) \in X_0
\end{cases}
\end{equation} 
in the energy space \textit{via} the backstepping modified method. 
\begin{proof}[Proof of Theorem \ref{main_int}]
Let $T>0$ and $\rho:=\rho_1>0$ such that $\|(\eta_0,w_0)\|_{\overline{X}_2}<\rho_1$, which will be given later. Consider the following system 
\begin{equation}\label{sss1}
\begin{cases}
\eta_{t} + w_{x}+w_{xxx}+(\eta w)_x= 0, & \text{in} \,\, (0,L)\times (0,T),\\
w_{t} +\eta_{x} +\eta_{xxx} + w w_{x} =0,  & \text{in} \,\, (0,L)\times (0,T), \\
\eta(0,t)=\eta(L,t)=0,& \text{on} \,\, (0,T),\\
\eta_{x}(0,t)=\displaystyle\int_0^L \left[ k_x(0,y)\eta(y,t)+s_x(0,y)w(y,t)\right]dy, & \text{on} \,\, (0,T),\\
w(0,t)=w(L,t)=0,&\text{on} \,\,  (0,T),\\w_x(L,t)=\displaystyle\int_0^L \left[ k_x(L,y)w(y,t)+s_x(L,y)\eta(y,t)\right]dy& \text{on} \,\, (0,T),\\
\eta(x,0)= \eta_0(x), \quad w(x,0)=  w_0(x), & \text{on} \,\, (0,L).
\end{cases}
\end{equation}
By Theorem \ref{nonlinearexitence1}, there exist $\rho_T$ such that the system \eqref{sss1} admits a unique solution $(\eta,w) \in C([0,T];\overline{X}_2)$ provided that $\|(\eta_0,w_0)\|_{\overline{X}_2}<\rho_T$. Moreover, there is a positive constant $C_T$ such that 
\begin{align}\label{eq23}
\|(\eta,w)\|_{C([0,T]; \overline{X}_2)} \leq C_T \|(\eta_0,w_0)\|_{\overline{X}_2}.
\end{align}
On the other hand, consider the transformations
\begin{equation}\label{trans}
u := (I-K)\eta-S w \quad \text{and} \quad v := (I-K)w-S \eta,
\end{equation}
where the operators $K$ and $S$ are given by \eqref{deftransf}. From the boundary conditions of \eqref{esta1.1} and \eqref{sss1}, we have that
\begin{align*}
u_{t}(x,t)=&\eta_{t}(x,t)-\int_0^Lk(x,y)\eta_{t}(y,t)dy-\int_0^Ls(x,y)w_{t}(y,t)dy \\
=&\eta_{t}(x,t)+\int_0^Lk(x,y)[w_{y}(y,t)+w_{yyy}(y,t)+(\eta w)_y]dy \\
&+ \int_0^Ls(x,y)[\eta_{y}(y,t)+\eta_{yyy}(y,t)+w w_{y}]dy \\
=&\eta_{t}(x,t)-\int_0^Lw(y,t)[k_{y}(x,y)+k_{yyy}(x,y)]dy-\int_0^L\eta(x,y)[s_{y}(x,y)+s_{yyy}(x,y)]dy \\
&- \int_0^Lk_{y}(x,y)\eta(y,t)w(y,t)dy-\frac12\int_0^Ls_{y}(x,y)w^2(y,t)dy,\\
u_{x}(x,t)=&u_{x}(x,t)=\eta_{x}(x,t)-\int_0^Lk_x(x,y)\eta(y,t)dy-\int_0^Ls_x(x,y)w(y,t)dy
\end{align*}
and
$$
u_{xxx}(x,t)=\eta_{xxx}(x,t)-\int_0^Lk_{xxx}(x,y)\eta(y,t)dy-\int_0^Ls_{xxx}(x,y)w(y,t)dy. 
$$
Similarly, we obtain 
\begin{align*}
v_{t}(x,t)=&w_{t}(x,t)-\int_0^L\eta(y,t)[k_{y}(x,y)+k_{yyy}(x,y)]dy-\int_0^Lw(x,y)[s_{y}(x,y)+s_{yyy}(x,y)]dy \\
&- \int_0^Ls_{y}(x,y)\eta(y,t)w(y,t)dy-\frac12\int_0^Lk_{y}(x,y)w^2(y,t)dy, \\
v_{x}(x,t)=&w_{x}(x,t)-\int_0^Lk_x(x,y)w(y,t)dy-\int_0^Ls_x(x,y)\eta(y,t)dy, 
\end{align*}
and
$$
v_{xxx}(x,t)=w_{xxx}(x,t)-\int_0^Lk_{xxx}(x,y)w(y,t)dy-\int_0^Ls_{xxx}(x,y)\eta(y,t)dy. 
$$
Thus, for a given $\lambda\in\R\setminus\{0\}$, it follows that
\begin{align*}
u_{t}(x,t)&+v_{x}(x,t)+v_{xxx}(x,t)+\lambda u(x,t) =[\eta_{t}(x,t)+w_{x}(x,t)+w_{xxx}(x,t)] \\
&-\int_0^Lw(y,t)[k_y(x,y)+k_{yyy}(x,y)+k_x(x,y)+k_{xxx}(x,y)+\lambda s(x,y)]dy \\
& -\int_0^L\eta(y,t)[s_y(x,y)+s_{yyy}(x,y)+s_x(x,y)+s_{xxx}(x,y)+\lambda k(x,y)-\lambda\delta(x-y)]dy \\
&- \int_0^Lk_{y}(x,y)\eta(y,t)w(y,t)dy-\frac12\int_0^Ls_{y}(x,y)w^2(y,t)dy \\
&=-(\eta w)_x- \int_0^Lk_{y}(x,y)\eta(y,t)w(y,t)dy-\frac12\int_0^Ls_{y}(x,y)w^2(y,t)dy
\end{align*}
and 
\begin{align*}
v_{t}(x,t)&+u_{x}(x,t)+u_{xxx}(x,t)+\lambda v(x,t) =[w_{t}(x,t)+\eta_{x}(x,t)+\eta_{xxx}(x,t)] \\
&-\int_0^L\eta(y,t)[k_y(x,y)+k_{yyy}(x,y)+k_x(x,y)+k_{xxx}(x,y)+\lambda s(x,y)]dy \\
& -\int_0^Lw(y,t)[s_y(x,y)+s_{yyy}(x,y)+s_x(x,y)+s_{xxx}(x,y)+\lambda k(x,y)-\lambda\delta(x-y)]dy \\
&- \int_0^Ls_{y}(x,y)\eta(y,t)w(y,t)dy-\frac12\int_0^Lk_{y}(x,y)w^2(y,t)dy \\
&=-ww_{x}- \int_0^Ls_{y}(x,y)\eta(y,t)w(y,t)dy-\frac12\int_0^Lk_{y}(x,y)w^2(y,t)dy.
\end{align*}
Hence, by using the system \eqref{esta1} and the boundary conditions \eqref{esta1.1}, we deduce that $(u,v)$ solves the system
\begin{equation}\label{ee1}
\begin{cases}
u_{t}(x,t)+v_{x}(x,t)+v_{xxx}(x,t)+\lambda u(x,t) =\sum_{i=1}^3\Psi_i(x,t), & \text{in $(0,L)\times[0,T]$}, \\
v_{t}(x,t)+u_{x}(x,t)+u_{xxx}(x,t)+\lambda v(x,t) =\sum_{i=1}^3\Phi_i(x,t), & \text{in $(0,L)\times[0,T]$}, \\
u(0,t)=u(L,t)=u_{x}(0,t)=0, & \text{on $[0,T]$},\\
v(0,t)=v(L,t)=v_{x}(L,t)=0, & \text{on $[0,T]$},
\end{cases}
\end{equation}
where 
\begin{equation*}
\begin{array}{l}
\Psi_1(x,t)= -(\eta(x,t) w(x,t))_x, \\
\displaystyle\Psi_2(x,t)=   - \int_0^Lk_{y}(x,y)\eta(y,t)w(y,t)dy, \\
\displaystyle\Psi_3(x,t) =-\frac12\int_0^Ls_{y}(x,y)w^2(y,t)dy
\end{array} 
\end{equation*}
and
\begin{equation*}
\begin{array}{l}
\Phi_1(x,t)= -w(x,t) w_{x}(x,t), \\
\displaystyle\Phi_2(x,t)=   - \int_0^Ls_{y}(x,y)\eta(y,t)w(y,t)dy, \\
\displaystyle\Phi_3(x,t) =-\frac12\int_0^Lk_{y}(x,y)w^2(y,t)dy.
\end{array}
\end{equation*}
Multiplying first equation of \eqref{ee1} by $u$, the second one by $v$ and integrating by parts on $(0,L)$, we get
\begin{multline}
\frac12\frac{d}{dt}\int_{0}^L(u^2(x,t)+v^2(x,t))dx \leq -\lambda \int_{0}^L(u^2(x,t)+v^2(x,t))dx \\ \label{eq16}
 +  \sum_{i=1}^3\left\lbrace \int_0^L \Psi_i (x,t)u(x,t)dx + \int_0^L \Phi_i (x,t)v(x,t)dx\right\rbrace.
\end{multline}

Now we will estimate the sum on the right hand side of \eqref{eq16}. First, note that, 

\begin{align*}
\int_0^L &\Psi_1(x) u(x)dx= - \int_0^L (\eta(x) w(x))_x\eta(x) dx + \int_0 ^L (\eta(x) w(x))_x \left( \int_0^L k(x,y)\eta_1(y)dy\right)dx \\
&+ \int_0 ^L (\eta(x) w(x))_x \left( \int_0^L s(x,y)w(y)dy\right)dx \\ 
= &  \int_0^L \eta(x) w(x)\eta_{x}(x) dx  - \int_0 ^L (\eta(x) w(x)) \left( \int_0^L k_x(x,y)\eta(y)dy\right)dx \\
&- \int_0 ^L (\eta(x) w(x)) \left( \int_0^L s_x(x,y)w(y)dy\right)dx \\ 
\leq & -\frac12 \int_0^L \eta^2(x) w_{x}(x) dx  + \|\eta(t)\|_{L^2(0,L)}\int_0 ^L |\eta(x)|| w(x)| \|k_x(x)\|_{L^2(0,L)}dx \\
&+ \|w(t)\|_{L^2(0,L)} \int_0 ^L |\eta(x)|| w(x) |\|s_x(x)\|_{L^2(0,L)}dx \\ 
\leq & \frac12\|w_{x}(t)\|_{L^{\infty}(0,L)} \|\eta(t)\|_{L^2(0,L)} ^2  + \sup_{x \in (0,L)} \|k_x(x)\|_{L^2(0,L)} \|\eta(t)\|^2_{L^2(0,L)} \|w(t)\|_{L^2(0,L)}\\
&+\sup_{x \in (0,L)} \|s_x(x)\|_{L^2(0,L)} \|w(t)\|^2_{L^2(0,L)} \|\eta(t)\|_{L^2(0,L)}.
\end{align*}
Therefore, 
\begin{align}\label{eq17}
\int_0^L \Psi_1(x,t)u(x,t)dx \leq K_0 \|(\eta(t),w(t))\|_{\overline{X}_2}\|(\eta(t),w(t))\|^2_{X_0}+ K_1\|(\eta(t),w(t))\|^3_{X_0},
\end{align}
where $K_0$ is the constant of the embedding $H^1(0,L)\subset L^{\infty}(0,L)$ and
\begin{equation*}
K_1 =  \sup_{x \in (0,L)} \|k_x(x)\|_{L^2(0,L)} + \sup_{x \in (0,L)} \|s_x(x)\|_{L^2(0,L)}.
\end{equation*}

In the same way, we have, 

\begin{align*}
\int_0^L &\Psi_2(x) u(x)dx= - \int_0^L \left( \int_0^L k_y(x,y)\eta(y)w(y)dy\right)\eta(x)dx  \\
&+ \int_0 ^L \left( \int_0^L k_y(x,y)\eta(y)w(y)dy\right) \left[  \int_0^L k(x,z)\eta(z)dz + \int_0^L s(x,z)w(z)dz \right] dx \\
=& - \int_0^L \eta(y)w(y)\left( \int_0^L k_y(x,y)\eta(x)dx\right)dy  \\
&+ \int_0 ^L\eta(y)w(y) \left( \int_0^L k_y(x,y)\left[ \int_0^L k(x,z)\eta(z)dz   + \int_0^L s(x,z)w(z)dz\right]dx\right)dy \\
\leq & \|\eta(t)\|_{L^2(0,L)} \int_0^L |\eta(y)||w(y)|\|k_y(y)\|_{L^2(0,L)}dy \\
& +\sup_{y \in (0,L)} \|k_y(y)\|_{L^2(0,L)} \|k\|_{L^2([0,L]\times[0,L])} \|\eta(t)\|_{L^2(0,L)}^2\|w(t)\|_{L^2(0,L)}\\
&+\sup_{y \in (0,L)} \|k_y(y)\|_{L^2(0,L)} \|s\|_{L^2([0,L]\times[0,L])} \|\eta(t)\|_{L^2(0,L)}\|w(t)\|_{L^2(0,L)}^2.
\end{align*}
Thus, 
\begin{align}\label{k2}
\int_0^L \Psi_2(x,t) u(x,t)dx \leq K_2\|(\eta(t),w(t))\|^3_{X_0},
\end{align}
where 
\begin{equation*}
K_2 =  \sup_{y \in (0,L)} \|k_y(y)\|_{L^2(0,L)}\left( 1+\|k\|_{L^2([0,L]\times[0,L])}+\|s\|_{L^2([0,L]\times[0,L])}\right).
\end{equation*}
Finally, 
\begin{align*}
\int_0^L &\Psi_3(x) u(x)dx= - \frac12\int_0^L \left( \int_0^L s_y(x,y)w^2(y)dy\right)\eta(x)dx  \\
&+ \frac12\int_0 ^L \left( \int_0^L s_y(x,y)w^2(y)dy\right)\left[ \int_0^L k(x,z)\eta(z)dz + \int_0^L s(x,z)w(z)dz\right]dx \\
=& - \frac12 \int_0^L w^2(y)\left( \int_0^L s_y(x,y)\eta(x)dx\right)dy  \\
&+ \frac12 \int_0 ^Lw^2(y) \left( \int_0^L s_y(x,y)\left[ \int_0^L k(x,z)\eta(z)dz   + \int_0^L s(x,z)w(z)dz\right]dx\right)dy \\
\leq & \frac12\|\eta(t)\|_{L^2(0,L)} \int_0^L w^2(y)\|s_y(y)\|_{L^2(0,L)}dy\\
&+ \frac12\|\eta(t)\|_{L^2(0,L)}\int_0 ^Lw^2(y) \left( \int_0^L |s_y(x,y)|\|k(x)\|_{L^2(0,L)}dx\right) dy \\
&+\frac12\sup_{y \in (0,L)} \|s_y(y)\|_{L^2(0,L)} \|k\|_{L^2([0,L]\times[0,L])} \|\eta(t)\|_{L^2(0,L)}\|w(t)\|_{L^2(0,L)}^2 \\
&+\frac12\sup_{y \in (0,L)} \|s_y(y)\|_{L^2(0,L)} \|s\|_{L^2([0,L]\times[0,L])} \|w(t)\|_{L^2(0,L)}^3.
\end{align*}
Then, 
\begin{align}\label{k3}
\int_0^L \Psi_3(x,t) u(x,t)dx \leq K_3\|(\eta(t),w(t))\|^3_{X_0},
\end{align}
where  
\begin{equation*}
K_3 =  \frac12\sup_{y \in (0,L)} \|s_y(y)\|_{L^2(0,L)}\left( 1+\|k\|_{L^2([0,L]\times[0,L])}+\|s\|_{L^2([0,L]\times[0,L])}\right).
\end{equation*}

Similarly, we can estimate the others three remaining terms on the right hand side of  \eqref{eq16}:
\begin{align}\label{eq21'}
\sum_{i=1}^3 \int_0^L \Phi_i (x,t)v(x,t)dx \leq  \left(\frac12 K_1 +\frac12K_2+ 2K_3\right)\|(\eta(t),w(t))\|^3_{X_0}.
\end{align}
Thus, by using \eqref{eq17}, \eqref{k2}, \eqref{k3} and \eqref{eq21'} in \eqref{eq16}, there exists a constant $\overline{K}$ such that
\begin{equation}\label{*}
\frac12\frac{d}{dt}\|(u(t),v(t))\|_{X_0}^2 \leq -\lambda \|(u(t),v(t))\|_{X_0}^2  
 +  \overline{K} \|(\eta(t),w(t))\|_{\overline{X}_2} \|(\eta(t),w(t))\|_{X_0}^2.
\end{equation}
On the other hand, note that \eqref{trans} can be rewrite as
\begin{align*}
u&=(I-(K+S))\eta+S(\eta-w) \\
v_1&=(I-(K+S))w-S(\eta-w),
\end{align*}
then, $u+v=(I-(K+S))(\eta+w)$ and $u-v=(I-(K-S))(\eta-w)$. Furthermore, due to Remark \ref{remarkinverse}, $(I-(K+S))^{-1}$ and $(I-(K-S))^{-1}$ belong to $\mathcal{L}(L^2(0,L))$. Thus, we can get that 
\begin{align*}
\|\eta(t)\|^2_{L^2(0,L)}&\leq \frac12\left\lbrace \|\eta(t)+w(t)\|^2_{L^2(0,L)}+\|\eta(t)-w(t)\|^2_{L^2(0,L)}\right\rbrace \\
&\leq C_1  \|(u(t),v(t))\|^2_{X_0},
\end{align*}
where $C_1$ is a positive constant given by 
\begin{align*}
C_1 = \frac{1}{2}\left\lbrace\|(I-(K+S))^{-1}\|^2_{\LL(L^2(0,L))} + \|(I-(K-S))^{-1}\|^2_{\LL(L^2(0,L))} \right\rbrace.
\end{align*}
Analogously, we obtain 
\begin{equation*}
\|w(t)\|^2_{L^2(0,L)}\leq C_1\|(u(t),v(t))\|^2_{X_0}.
\end{equation*}
Hence, there exists $C_1=C_1(K,S)>0$ satisfying  
\begin{equation}\label{eq22}
\|(\eta(t),w(t))\|^2_{X_0}\leq C_1  \|(u(t),v(t))\|^2_{X_0}.
\end{equation}
By \eqref{*}-\eqref{eq22}, it follows that 
\begin{equation*}
\frac12\frac{d}{dt}\|(u(t),v(t))\|_{X_0}^2 \leq -\left(\lambda- \overline{K}C_1 \|(\eta(t),w(t))\|_{\overline{X}_2} \right) \|(u(t),v(t))\|_{X_0}^2. 
\end{equation*}
For a given $\lambda > 0$, we know that there is  $\delta_1 > 0$ such that, if $\|(\eta(0),w(0))\|_{\overline{X}_2} <\delta_1$, we have
\begin{align*}
\overline{K}C_1 \|(\eta(t),w(t))\|_{\overline{X}_2}  <\frac{\lambda}{2}, \quad \forall t \in [0,T].
\end{align*}
Thus, we have  
\begin{equation*}
\frac{d}{dt}\|(u(t),v(t))\|_{X_0}^2 \leq -\lambda \|(u(t),v(t))\|_{X_0}^2, \quad \forall t \in [0,T],
\end{equation*}
which implies that 
\begin{equation*}
\|(u(t),v(t))\|_{X_0} \leq e^{-\frac{\lambda}{2} t} \|(u(0),v(0))\|_{X_0}, \quad \forall t \in [0,T].
\end{equation*}
Consequently, using the expression of $u$ and $v$, defined by \eqref{trans}, in \eqref{eq22}, we get that
\begin{equation*}
\|(\eta(t),w(t))\|_{X_0} \leq C e^{-\frac{\lambda}{2} t} \|(\eta_0,w_0)\|_{X_0}, \quad \forall t \in [0,T],
\end{equation*}
for some positive constant $C=C(K,S)$. Therefore, the proof of the theorem is finished.
\end{proof}

\section{Further comments and open problems}\label{FC}



\vspace{0.3cm}
\noindent$\bullet$\textit{ \textbf{One control on the right end-point}}
\vspace{0.2cm}

If we consider homogeneous Dirichlet condition and one control inputs at the Neumann boundary condition, then we are not able to prove the rapid stabilization via the  backstepping.  For instance, if we take $g=0$, the boundary condition \eqref{b1.1} becomes 
\begin{equation}\label{bbb1.1}
\begin{cases}
\eta(0,t)=0,\,\,\eta(L,t)=0,\,\,\eta_{x}(0,t)=f(t),&t \in (0,\infty), \\
w(0,t)=0,\,\,w(L,t)=0,\,\,w_{x}(L,t)=0,& t \in (0,\infty).
\end{cases}
\end{equation}
As we did before, a natural idea is to use the transformation
\begin{equation}\label{eq30}
\begin{cases}
u(x,t) = \eta(x,t)- \displaystyle\int^L_0k(x,y)\eta(y,t)dt-\int_0^Ls(x,y) w(y,t)dy \\
 v(x,t) = w(x,t)-\displaystyle\int^L_0k(x,y)w(y,t)dy - \int^L_0s(x,y)\eta(y,t)dy,
\end{cases}
\end{equation}
where $(k(\cdot,\cdot), s(\cdot,\cdot))$ is a solution of an appropriate  stationary system. However, it is not clear if that approach, used in this paper, works in this case. Indeed, \eqref{eq30} implies that the feedback law will be given by 
\begin{equation}\label{feedback2}
f(t):=F(\eta,w)=\displaystyle\int^L_0k_x(0,y)\eta(y,t)dt+\int_0^Ls_x(0,y) w(y,t),
\end{equation}
where $(k(\cdot,\cdot), s(\cdot,\cdot))$ should solves the stationary system \eqref{esta1} with boundary condition \eqref{esta1.1} and 
a additional restriction $$k_x(L,\cdot)=s_x(L,\cdot)=0.$$ As in Theorem \ref{mod_coron}, is not difficult to see that the well-posedness of the above stationary problem is equivalent to the well-posedness of the following problem
\begin{equation}\label{eq31}
\begin{cases}
\overline{k}_{yyy}+\overline{k}_y+\overline{k}_{xxx}+\overline{k}_x \pm\lambda \overline{k}= \pm\lambda \delta (x-y) & \text{in $(0,L)\times(0,L)$}, \\
\overline{k}(x,0)=\overline{k}(x,L)=0, & \text{on $(0,L)$}, \\
\overline{k}_y(x,0)=\overline{k}_y(x,L)=0, & \text{on $(0,L)$}, \\
\overline{k}(0,y)=\overline{k}(L,y)=0, & \text{on $(0,L)$}, \\
\overline{k}_x(L,y)=0, & \text{on $(0,L)$},
\end{cases}
\end{equation}
for any $\lambda \in \R \setminus \{0\}$. However, with these boundary restrictions the third order system \eqref{eq31} becomes over-determined. Therefore,  is not clear if such function $\overline{k}(\cdot,\cdot)$ exists. Thus, the natural open problem appears:

\vspace{0.2cm}
\noindent\textit{{\bf Question $\mathcal{A}$:}
Can we prove that the nonlinear system \eqref{n1'''}-\eqref{bbb1.1}, with $f(t)$ defined by \eqref{feedback2}, is exponential stable, by using the backstepping method?}

\vspace{0.4cm}
\noindent$\bullet$\textit{ \textbf{Less regularity of the initial data}}
\vspace{0.2cm}

Consider the following linear KdV-KdV system
\begin{equation}
\label{lin_f}
\begin{cases}
\eta_t + w_x+w_{xxx}= 0, & \text{in} \,\, (0,L)\times (0,+\infty),\\
w_t +\eta_x +\eta_{xxx}=0,  & \text{in} \,\, (0,L)\times (0,+\infty), \\
\eta(x,0)= \eta_0(x), \quad w(x,0)=  w_0(x), & \text{in} \,\, (0,L),
\end{cases}
\end{equation}
with following boundary conditions
\begin{equation}\label{lin_f1}
\begin{cases}
\eta(0,t)=0,\,\,\eta(L,t)=0,\,\,\eta_{x}(0,t)=f(t),& \text{on } (0,T), \\
w(0,t)=0,\,\,w(L,t)=0,\,\,w_{x}(L,t)=g(t).& \text{on } (0,T),
\end{cases}
\end{equation}
where $f(t):=F_1(\eta(t),w(t))$ and $g(t):= F_2(\eta(t),w(t))$, with $F_i$ in $\LL(X_0;\mathbb{R})$, $i=1,2$,  defined by \eqref{feedback1}. 

When we required less regularity of the initial data, the following result of the well-posedness for the system \eqref{lin_f}-\eqref{lin_f1} is true:

\vspace{0.1cm} 
\noindent\textit{For every $(\eta_0,w_0)\in X_{-1}$ and $(f,g)\in [L^2(0,T)]^2$, there exists a unique solution $$(\eta, w)\in C([0,T];X_{-1})$$ of system \eqref{lin_f}-\eqref{lin_f1}, such that 
\begin{align*}
\|(\eta,w)\|_{L^{\infty}(0,T;X_{-1})} \leq C \left( \|(\eta_0,w_0)\|_{X_{-1}} + \|(f,g)\|_{[L^2(0,T)]^2}\right),
\end{align*}
for some positive constant $C=C(T)$.}

\vspace{0.1cm} 
To prove it, we use the classical approach given by the Riesz representation Theorem to obtain a solution by transposition, see \cite{capis_thesis,capistrano2016} for more details.  Our analysis on the case of regular data (Theorem \ref{main_int}) suggests that is possible to obtain the rapid exponential stabilization for the linear system \eqref{lin_f} for less regularity of the initial data whenever the linear system is well-posedness in some sense on $X_0$.  Thus, another natural question arises here is the following one:

\vspace{0.2cm}
\noindent\textit{{\bf Question $\mathcal{B}$:}
Is the linear system \eqref{lin_f}-\eqref{lin_f1}, with $f(t):=F_1(\eta(t),w(t))$ and $g(t):= F_2(\eta(t),w(t))$ defined by \eqref{feedback1}, well-posedness in $X_0$ for less regular initial data $(\eta_0,w_0)$?}

\vspace{0.1cm}
{Due to a lack of a priori $X_0$--estimate, the issue of the rapid stabilization for the nonlinear problem is difficult to address.  Only when the initial data is regular we can get a positive answer to this question, as was proved in Theorem \ref{main_int}. Indeed, in order to obtain a appropriate  bound for \eqref{eq17},  we have to estimate the term $-\frac12 \int_0^L \eta^2(x) w_{x}(x) dx$. This explain why  the estimation in the $\overline{X}_2$-norm it is necessary in our approach.  Therefore, for initial data in $X_0$, the following question remains open:

\vspace{0.2cm}
\noindent\textit{{\bf Question $\mathcal{C}$:}
Is the nonlinear system \eqref{n1'''}-\eqref{lin_f1}, with $f(t):=F_1(\eta(t),w(t))$ and $g(t):= F_2(\eta(t),w(t))$ defined by \eqref{feedback1}, exponential stable for initial data $(\eta_0,w_0)\in X_0$?}

\vspace{0.4cm}
\noindent$\bullet$\textit{ \textbf{Global well-posedness in time}}
\vspace{0.2cm}

\vspace{0.2cm}

Adapting the proof of Theorem \ref{nonlinearexitence1}, one can also prove that, without any restriction over the initial data $(\eta_0,w_0)$, there exist $T^* > 0$ and a solution $(\eta,w)$ of \eqref{n1'''}-\eqref{n1.1''}, with the feedback law $f(t) = F_1(\eta(\cdot,t),w(\cdot,t))$ and $g(t)= F_1(\eta(\cdot,t),w(\cdot,t))$ satisfying the initial condition $\eta(\cdot,0) = u_0(\cdot)$ and $w(\cdot,0)=w_0(\cdot)$. More precisely,

\begin{thm}
Let $F_i:X_0\longrightarrow \R$ be a continuous linear map for $i=1,2$ and the initial data $(\eta_0,w_0) \in \overline{X}_2$.  Then, there exists $T^*>0$ such that, there exists a unique solution $(\eta, w) \in C([0,T^*]; \overline{X}_2)$ of \eqref{n1'''}-\eqref{n1.1''} with $f(t):=F_1(\eta(t),w(t))$, $g(t):=F_2(\eta(t),w(t))$. Moreover
\begin{align*}
\|(\eta,w)\|_{C([0,T]; \overline{X}_2)} \leq C \|(\eta_0,w_0)\|_{ \overline{X}_2},
\end{align*}
for some positive constant $C=C(T^*)$.
\end{thm}

Observe that if $(\eta_1, w_1)\in C([0,T_1],\overline{X}_2)$ and $(\eta_2, w_2)\in C([0,T_2],\overline{X}_2)$  are  the solutions given by the Theorem \ref{nonlinearexitence1} with initial data $(\eta_0, w_0)$ and $(\eta_1(T_1), w_1(T_1))$, respectively, the function $(\eta,w):[0,T_1+T_2]\rightarrow \overline{X}_2$ defined by 
\begin{align*}
(\eta(t),w(t))=\begin{cases}
(\eta_1(t), w_1(t))  & \text{if $t \in [0,T_1]$}, \\
(\eta_2(t-T_1), w_2(t-T_2))  & \text{if $t \in [T_1,T_1+T_2]$}, 
\end{cases}
\end{align*}
is the solution of the feedback system on interval $[0,T_1+T_2]$ with initial data $(\eta_0, w_0)$. This argument allows us extend a local solution until a maximal interval, that is, for all $0<T<T_{\max} \leq \infty$ there exist a function $(\eta, w)\in C([0,T],\overline{X}_2)$ solution of the feedback system \eqref{n1'''}-\eqref{n1.1''}. The following proposition, easily holds:
\begin{prop}
Let $(\eta_0, w_0) \in \overline{X}_2$ and $(\eta, w)\in C([0,T],\overline{X}_2)$ solution of the feedback system, for all $0<T<T_{\max}$, with initial data $(\eta_0, w_0)$. Then, only one of the following assertions hold:
\begin{enumerate}
\item[(i)] $T_{\max}=\infty$;
\item[(ii)] If $T_{\max}<\infty$, then, $\lim_{t\rightarrow T_{\max}}\|(\eta(t),w(t))\|_{\overline{X}_2}=\infty$. 
\end{enumerate}
\end{prop}
Thus, the following questions related to the global well-posedness in time are also important:

\vspace{0.2cm}
\noindent\textit{{\bf Question $\mathcal{D}$:}
Let $F_i:X_0\longrightarrow \R$ be a continuous linear map for $i=1,2$. Is the nonlinear system \eqref{n1'''}-\eqref{n1.1''}, with $f(t):=F_1(\eta(t),w(t))$ and $g(t):= F_2(\eta(t),w(t))$, global well-posedness in time, i.e, $T_{\max}$ is infinity?}

\vspace{0.2cm}
\noindent\textit{{\bf Question $\mathcal{E}$:}
If the question $\mathcal{D}$ has positive answer, do have expect some restriction on the initial data?}

\vspace{0.4cm}
\noindent\textbf{Acknowledgments:} This work was carried out during the visit of the second author to the Federal University of Pernambuco. F. A. Gallego would like to thank the Mathematics Department at Federal University of Pernambuco, in Recife, for its hospitality.

\end{document}